\pdfoutput=1
\RequirePackage{ifpdf}
\ifpdf 
\documentclass[pdftex]{sigma}
\else
\documentclass{sigma}
\fi

\def\C{{\mathbb C}}
\def\P{{\mathbb P}}
\def\dfrac#1#2{{\displaystyle\frac{#1}{#2}}}
\numberwithin{equation}{section}

\newtheorem{Theorem}{Theorem}[section]
\newtheorem{Corollary}[Theorem]{Corollary}
\newtheorem{Lemma}[Theorem]{Lemma}
\newtheorem{Proposition}[Theorem]{Proposition}
 { \theoremstyle{definition}
\newtheorem{Definition}[Theorem]{Definition}
\newtheorem{Example}[Theorem]{Example}
\newtheorem{Remark}[Theorem]{Remark} }

\begin{document}
\allowdisplaybreaks

\newcommand{\arXivNumber}{2104.06661}

\renewcommand{\PaperNumber}{076}

\FirstPageHeading

\ShortArticleName{Quantum Representation of Affine Weyl Groups and Associated Quantum Curves}

\ArticleName{Quantum Representation of Affine Weyl Groups\\ and Associated Quantum Curves}

\Author{Sanefumi MORIYAMA~$^{\rm a}$ and Yasuhiko YAMADA~$^{\rm b}$}

\AuthorNameForHeading{S.~Moriyama and Y.~Yamada}

\Address{$^{\rm a)}$~Department of Physics/OCAMI/NITEP, Osaka City University,\\
\hphantom{$^{\rm a)}$}~Sugimoto, Osaka 558-8585, Japan}
\EmailD{\href{mailto:sanefumi@osaka-cu.ac.jp}{sanefumi@osaka-cu.ac.jp}}

\Address{$^{\rm b)}$~Department of Mathematics, Kobe University, Rokko, Kobe 657-8501, Japan}
\EmailD{\href{mailto:yamaday@math.kobe-u.ac.jp}{yamaday@math.kobe-u.ac.jp}}

\ArticleDates{Received May 13, 2021, in final form August 04, 2021; Published online August 15, 2021}

\Abstract{We study a quantum (non-commutative) representation of the affine Weyl group mainly of type $E_8^{(1)}$, where the representation is given by birational actions on two variables~$x$,~$y$ with $q$-commutation relations. Using the tau variables, we also construct quantum ``fundamental'' polynomials $F(x,y)$ which completely control the Weyl group actions. The geometric properties of the polynomials $F(x,y)$ for the commutative case is lifted distinctively in the quantum case to certain singularity structures as the $q$-difference operators. This property is further utilized as the characterization of the quantum polynomials $F(x,y)$. As an application, the quantum curve associated with topological strings proposed recently by the first named author is rederived by the Weyl group symmetry. The cases of type $D_5^{(1)}$, $E_6^{(1)}$, $E_7^{(1)}$ are also discussed.}


\Classification{39A06; 39A13}

\section{Introduction}\label{section1}

Quantization of the Painlev\'e equations (or isomonodromic deformations more generally) and their discrete variations is an important problem.
Recently, this subject attracts various interests due to its relation to conformal field theories, gauge theories and topological strings.
Despite some interesting pioneering works~\cite{BGM,BS,BGT17,CMR,Hasegawa,Kuroki}, there remain many problems to be studied especially on the quantization of the discrete Painlev\'e equations.
One of the main problems is to establish the quantization compatible with the geometric formulation in~\cite{KNY,Sakai:2001}.\footnote{In Appendix~\ref{sectionB}, we give a short summary for the classical cases.}
Such a~study is expected to clarify various developments mentioned above from a geometric viewpoint of {\it quantum curves}.

Recently, in the study of topological strings, certain quantum curves related to the affine Weyl group of type $D_5^{(1)}$, $E_6^{(1)}$, $E_7^{(1)}$, $E_8^{(1)}$ were obtained~\cite{Moriyama}.
The quantum curves were obtained by~combining previous classical results in~\cite{BBT,KY} and an empirical observation for quantization of the classical multiplicities~\cite{KMN} (as discussed later in Section~\ref{section3}).
Our main motivation is to formulate a quantum representation of the affine Weyl groups to provide a solid basis for the study of these quantum curves and the corresponding quantum $q$-difference Painlev\'e equations.
Among others, our work enables the derivation of these quantum curves from the first principle.\footnote{Recently, the elliptic quantum curve for the E-string theory is obtained in~\cite{CHKSW}.}
As discussed in the last section, we expect that our work will clarify the group structure of~various related physical theories.
In particular, we hope to relate directly our tau functions fully equipped with the group structure to the partition functions in topological strings in the future.

The contents of this paper is as follows.
In the remaining part of this section, we recall some basic results on the representation of the affine Weyl group $W\big(E_8^{(1)}\big)$ in the commutative case, focusing on polynomials (which we call fundamental or $F$-polynomials) generated by the Weyl group actions.
In Section~\ref{section2}, a natural quantization of the representation of $W\big(E_8^{(1)}\big)$ is formulated.
The quantization of the $F$-polynomials is associated to $q$-difference operators and we study a crucial non-logarithmic property of it in Section~\ref{section3}.
In Section~\ref{section4}, we show the main theorem which characterizes the quantum $F$-polynomials.
In Section~\ref{section5}, applying the constructions, we give a characterization of the quantum curve of type $E_8$.
In Section~\ref{section6}, we give a bilinear form of the Weyl group actions.
Section~\ref{section7} is for summary and discussions.
In~Appendix~\ref{sectionA}, the similar constructions are obtained for the cases of $D_5^{(1)}$, $E_6^{(1)}$ and $E_7^{(1)}$.
In~Appendix~\ref{sectionB}, the relation of the classical Weyl group representation in Section~\ref{section1} to the standard representations used in the $q$-Painlev\'e equations is summarized.

In order to explain the problem of this paper more explicitly, we recapitulate some basic facts on a birational representation of the affine Weyl group of type $E_8^{(1)}$, $W\big(E_8^{(1)}\big)=\langle s_0, s_1, \dots, s_8 \rangle$ defined by the Dynkin diagram:
\begin{gather*}
\begin{array}{@{}c@{\ }c@{\ }c@{\ }c@{\ }c@{\ }c@{\ }c@{\ }c@{\ }c@{\ }c@{\ }c@{\ }c@{\ }c@{\ }c@{\ }c@{}}
&&&&s_0\\
&&&&|\\
s_1&\text{---}&s_2&\text{---}&s_3&\text{---}&s_4&\text{---}&s_5&\text{---}&s_6&\text{---}&s_7&\text{---}&s_8.
\end{array}
\end{gather*}
All the results in this section are known in literature (see~\cite{Tsuda} for example) up to a change of~para\-metrization, hence we omit the proofs.

\begin{Proposition}\label{prop:classical_action}
Define the algebra automorphism $s_0, \dots, s_8$ on parameters $h_1,h_2, e_1, \dots, e_{11}$ and variables $x, y, \sigma_1, \sigma_2, \tau_1, \dots, \tau_{11}$ as
\begin{gather}
s_0=\bigg\{e_{10}\to \dfrac{h_2}{e_{11}},\, e_{11}\to \dfrac{h_2}{e_{10}},\, h_1\to \dfrac{h_1 h_2}{e_{10} e_{11}},\,x\to x\dfrac{1+y\frac{h_2}{e_{10}}}{1+ye_{11}}, \nonumber
\\ \hphantom{s_0=\bigg\{}
\tau_{10}\to \big(1+ye_{11}\big)\dfrac{\sigma_2}{\tau_{11}},\,
\tau_{11}\to \dfrac{\sigma_2}{\tau_{10}}\bigg(1+y\dfrac{h_2}{e_{10}}\bigg),\,
\sigma_1\to \big(1+ye_{11}\big)\dfrac{\sigma_1 \sigma_2}{\tau_{10} \tau_{11}} \bigg\}, \nonumber
\\[.5ex]
s_1=\{e_8\leftrightarrow e_9,\, \tau_8\leftrightarrow \tau_9\}, \quad\
s_2=\{e_7\leftrightarrow e_8,\, \tau_7\leftrightarrow \tau_8\}, \nonumber
\\[.5ex]
s_3=\bigg\{e_1\to \dfrac{h_1}{e_7},\, e_7\to \dfrac{h_1}{e_1},\, h_2\to \dfrac{h_1 h_2}{e_1 e_7},\,
y\to \dfrac{1+x\frac{e_7}{h_1}}{1+\frac{x}{e_1}}y, \nonumber
\\ \hphantom{s_3=\bigg\{}
\tau_1\to \bigg(1+x\dfrac{e_7}{h_1}\bigg)\dfrac{\sigma_1}{\tau_7},\,
\tau_7\to \dfrac{\sigma_1 }{\tau_1}\bigg(1+\dfrac{x}{e_1}\bigg),\,
\sigma_2\to \dfrac{\sigma_1 \sigma_2}{\tau_1 \tau_7}\bigg(1+\dfrac{x}{e_1}\bigg)\bigg\}, \nonumber
\\[.5ex]
s_4=\{e_1\leftrightarrow e_2,\, \tau_1\leftrightarrow \tau_2\}, \quad\
s_5=\{e_2\leftrightarrow e_3,\, \tau_2\leftrightarrow \tau_3\}, \quad\
s_6=\{e_3\leftrightarrow e_4,\, \tau_3\leftrightarrow \tau_4\}, \nonumber
\\[.5ex]
s_7=\{e_4\leftrightarrow e_5,\, \tau_4\leftrightarrow \tau_5\}, \quad\
s_8=\{e_5\leftrightarrow e_6,\, \tau_5\leftrightarrow \tau_6\}.
\label{eq:classical_action}
 \end{gather}
Then these actions give a birational representation of the affine Weyl group $W\big(E^{(1)}_8\big)$ on the field of rational functions $\C(h_i, e_i, x, y, \sigma_i, \tau_i)$.
\end{Proposition}

The representation is based on a special configuration of $11$ points on $\P^1\times \P^1$ (see Figure~\ref{fig:11config}).
For the blow-up $X$ of $\P^1 \times \P^1$ at the 11 points $p_i$ $(i=1,2,\dots,11)$, the Picard lattice $P={\rm Pic}(X)$ is generated by $H_1,H_2, E_1, \dots, E_{11}$, with the only non-vanishing intersection pairings being $H_1\cdot H_2=H_2\cdot H_1=1$, $E_i\cdot E_i=-1$.
The actions (\ref{eq:classical_action}) are closed on subfields $\C(h_i,e_i)$ and~$\C(h_i,e_i,x,y)$.
The restriction on $\C(h_i,e_i)$
\begin{gather*}
s_0=\bigg\{e_{10}\to \dfrac{h_2}{e_{11}},\ e_{11}\to \dfrac{h_2}{e_{10}},\ h_1\to \dfrac{h_1 h_2}{e_{10} e_{11}}\bigg\},
\qquad
s_1=\{e_8\leftrightarrow e_9\}, \qquad
s_2=\{e_7\leftrightarrow e_8\},
\\
s_3=\bigg\{e_1\to \dfrac{h_1}{e_7},\ e_7\to \dfrac{h_1}{e_1},\ h_2\to \dfrac{h_1 h_2}{e_1 e_7}\bigg\}, \qquad
s_4=\{e_1\leftrightarrow e_2\}, \qquad
s_5=\{e_2\leftrightarrow e_3\},
\\
s_6=\{e_3\leftrightarrow e_4\}, \qquad
s_7=\{e_4\leftrightarrow e_5\}, \qquad
s_8=\{e_5\leftrightarrow e_6\},
\end{gather*}
is nothing but the natural linear actions on the Picard lattice written in the multiplicative notation: $h_i =\exp H_i$, $e_i=\exp E_i$.
When $x=y=0$, the actions on $\sigma_i$, $\tau_i$ are just copies of the actions on $h_i$, $e_i$.
In terms of the parameters $h_i$, $e_i$ the points $p_1, \dots, p_{11}$ can be parametrized as
\begin{gather*}
p_i=\big({-}e_i,0\big) \qquad (i=1, \dots, 6),\qquad
p_i=\bigg({-}\dfrac{h_1}{e_i}, \infty\bigg) \qquad (i=7, 8, 9),
\\
p_{10}=\bigg(\infty,-\dfrac{e_{10}}{h_2}\bigg), \qquad
p_{11}=\bigg(0, -\dfrac{1}{e_{11}}\bigg).
\end{gather*}
This parametrization is compatible under the actions of the Weyl group $W\big(E_8^{(1)}\big)$.
\begin{figure}[h]
\begin{center}
\setlength{\unitlength}{0.7mm}
\begin{picture}(90,45)(10,5)
\put(0,10){\line(1,0){90}}
\put(0,40){\line(1,0){90}}
\put(10,0){\line(0,1){50}}
\put(80,0){\line(0,1){50}}
\put(5,-5){$x=0$}\put(75,-5){$x=\infty$}
\put(95,10){$y=0$}\put(95,40){$y=\infty$}
\put(10,25){\circle*{2} $p_{11}$}
\put(80,25){\circle*{2} $p_{10}$}
\put(20,10){\circle*{2}}\put(20,13){$p_1$}
\put(30,10){\circle*{2}}\put(30,13){$p_2$}
\put(40,10){\circle*{2}}\put(40,13){$p_3$}
\put(50,10){\circle*{2}}\put(50,13){$p_4$}
\put(60,10){\circle*{2}}\put(60,13){$p_5$}
\put(70,10){\circle*{2}}\put(70,13){$p_6$}
\put(25,40){\circle*{2}}\put(25,43){$p_7$}
\put(45,40){\circle*{2}}\put(45,43){$p_8$}
\put(65,40){\circle*{2}}\put(65,43){$p_9$}
\end{picture}
\end{center}
\caption{Configuration of the 11 points.}
\label{fig:11config}
\end{figure}
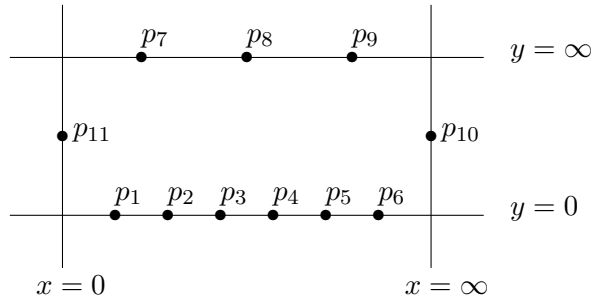

For an algebraic curve in $X$, its homological data $\lambda=(d_i, m_i)$ (i.e., the bidegree $(d_1,d_2)$ and the multiplicity $m_i$ at the $i$-th point $p_i$) can be represented by an element of $P$ as
\begin{gather}\label{eq:add_lambda}
\lambda=d_1 H_1+d_2 H_2-m_1 E_1-\cdots -m_{11}E_{11}.
\end{gather}
Sometimes, to represent the data $\lambda=(d_i, m_i)$, we use a multiplicative notation
\begin{gather*}
e^\lambda=\dfrac{h_1^{d_1}h_2^{d_2}}{e_1^{m_1}\cdots e_{11}^{m_{11}}},\qquad
\tau^\lambda=\dfrac{\sigma_1^{d_1}\sigma_2^{d_2}}{\tau_1^{m_1}\cdots \tau_{11}^{m_{11}}}.
\end{gather*}

We call variables $\sigma_i$, $\tau_i$ the {\it tau} variables (or {\it tau} functions).
The tau functions are the main objects in the theory of isomonodromic deformations~\cite{JM}, and their representation-theoretical formulation was initiated in~\cite{NY:cmp}.
Although quantum curves as well as their classical analogs can be discussed in the subfield $\C(h_i,e_i,x,y)$ as in~\cite{Moriyama}, we stress that the appropriate introduction of the variables $\sigma_i$, $\tau_i$ in equation~(\ref{eq:classical_action}) clarifies the structure of the Weyl group actions largely since it reduces the problem of rational functions of $x$, $y$ into that of polynomials (see the last remark in this section).
Indeed, the basic fact on the representation (\ref{eq:classical_action}) is the following holomorphic property which is related to the {\it singularity confinement} (see~\cite{GNR} and references therein) and the {\it Laurent phenomenon}~\cite{FZ}.
\begin{Proposition}\label{phi_lambda}
For any $w \in W\big(E^{(1)}_8\big)$, the action of $w$ on variables $\tau_i$ $(i=1, \dots,11)$ is given~by
\begin{gather*}
w(\tau_i)=\phi_{w, i}(x,y) \tau^\lambda, \qquad
\tau^\lambda=\frac{\sigma_1^{d_1}\sigma_2^{d_2}}{\tau_1^{m_1}\cdots \tau_{11}^{m_{11}}},
\end{gather*}
where $\lambda=(d_i, m_i)$ is determined by
$w(e_i)=e^\lambda=\frac{h_1^{d_1}h_2^{d_2}}{e_1^{m_1}\cdots e_{11}^{m_{11}}}$, and $\phi_{w, i}(x,y)$ is a polynomial associated with the degree/multiplicity data $\lambda=(d_i, m_i)$.
Moreover, regardless of the above construction using the action of $w$, the polynomial $\phi_{w, i}(x,y)$ can be recovered by the geometric conditions specified by the data $\lambda=(d_i, m_i)$ uniquely up to a normalization.
Hence we can denote $\phi_{w,i}(x,y)$ by $F_\lambda(x,y)$.
\end{Proposition}

\begin{Remark}
The curves $C\colon F_\lambda(x,y)=0$ are transforms of the exceptional curve $E_i$ under the birational actions $w \in W\big(E_8^{(1)}\big)$, hence the curves $C$ are rational and rigid.
\end{Remark}

\begin{Example}
For $w=s_{3,2,1,0,2,4,3}$ $(=s_3s_2s_1s_0s_2s_4s_3)$, we have $e^\lambda:=w(e_1)=\frac{h_1^2 h_2}{e_1 e_7 e_9 e_{10} e_{11}}$, and
\begin{gather}\label{eq:ex1_cl}
F_{\lambda}(x,y)=
\bigg(1+\frac{e_1 e_7 e_9 e_{10} e_{11}}{h_1^2 h_2} x\bigg) \bigg(1+\frac{1}{e_1}x\bigg)
+e_{11}\bigg(1+\frac{e_7}{h_1} x\bigg) \bigg(1+\frac{e_9}{h_1} x\bigg)y.
\end{gather}
For $w=s_{0, 3, 4, 0, 2, 3, 2, 1, 0, 2, 4, 3}$, we have $e^\lambda:=w(e_{11})=\frac{h_1^2 h_2^2}{e_1 e_2 e_7 e_8 e_{10}^2 e_{11}}$, and
\begin{align}
F_{\lambda}(x,y)={}&
\frac{x^2 \big(1+\frac{h_2}{e_{10}} y\big)^2}{e_1 e_2}
+x\bigg(1+\frac{h_2}{e_{10}} y\bigg) \bigg\{\bigg(\frac{1}{e_7}+\frac{1}{e_8}\bigg) \frac{h_1 h_2}{e_1e_2e_{10}} y+\bigg(\frac{1}{e_1}+\frac{1}{e_2}\bigg)\bigg\}\nonumber
\\[1ex]
 &+\big(1+e_{11} y\big) \bigg(1+\frac{h_1^2 h_2^2}{e_1 e_2 e_7 e_8 e_{10}^2 e_{11}} y\bigg).
\end{align}\label{eq:ex2_cl}
\end{Example}

\begin{Remark}
We see that the variables $k_1$, $k_2$ defined by
\begin{gather}\label{eq:k1k2}
k_1=x \dfrac{\tau_{10}}{\tau_{11}}, \qquad
k_2=y \dfrac{\tau_7\tau_8\tau_9}{\tau_1\tau_2\cdots \tau_6},
\end{gather}
are $W\big(E^{(1)}_8\big)$ invariant.
Hence, the rational actions of $w \in W\big(E^{(1)}_8\big)$ on $x$, $y$ can be determined by the polynomials corresponding to $w(\tau_1), \dots, w(\tau_{11})$.
\end{Remark}

\section{Quantum representation}\label{section2}

In the following, we use the same symbols $h_i$, $e_i$, $x$, $y$, $\sigma_i$, $\tau_i$ for the quantum (non-commutative) objects.
This notation is economical and consistent with the commutative case since the latter can be recovered by taking the specialization $q=1$.

\begin{Definition}
Let ${\mathcal K}$ be a skew (non-commutative) field on the variables $h_1,h_2, e_1, \dots, e_{11}$, $x,y, \sigma_1, \sigma_2, \tau_1, \dots, \tau_{11}$, where the non-trivial commutation relations are
\begin{gather}\label{eq:qcom}
yx=qxy, \qquad
\tau_i e_i=q^{-1}e_i\tau_i, \qquad
\sigma_1 h_2=q h_2\sigma_1, \qquad
\sigma_2 h_1=q h_1\sigma_2,
\end{gather}
and other pairs are assumed to be commutative.
\end{Definition}

\begin{Remark}
In view of the results in~\cite{Kuroki2}, where the construction of~\cite{NY:birat} is nicely quantized, it is natural to regard the variables $\sigma_i$, $\tau_i$ to be dual to the parameters $h_i$, $e_i$.
Indeed, the $q$-commutation relations (\ref{eq:qcom}) can be concisely written as
\begin{gather*}
\tau^{\lambda}e^\mu=q^{\lambda\cdot \mu}e^\mu \tau^\lambda,
\end{gather*}
using the intersection pairing $\lambda\cdot \mu=d_1d'_2+d_2d'_1-\sum_{i=1}^{11}m_i m'_i$ for
$\tau^\lambda={\sigma_1^{d_1}\sigma_2^{d_2}}/\big({\tau_1^{m_1}\cdots \tau_{11}^{m_{11}}}\big)$,
$e^{\mu}={h_1^{d'_1}h_2^{d'_2}}/\big(e_1^{m'_1}\cdots e_{11}^{m'_{11}}\big)$ as well as $\lambda=(d_i, m_i)$, $\mu=(d'_i, m'_i)$.
\end{Remark}

Under the non-commutative setting given above, there exists a natural quantization of Proposition \ref{prop:classical_action}.
\begin{Theorem}\label{thm:quantum-rep}
On the skew field ${\mathcal K}$ there exists a birational representation of the affine Weyl group $W\big(E^{(1)}_8\big)=\langle s_0, \dots, s_8\rangle$ given exactly by the same equation as in equation~\eqref{eq:classical_action}.
\end{Theorem}

\begin{proof}
A direct computation (see also Remark~\ref{remark6}).
\end{proof}

\begin{Remark}
We have fixed the operator ordering in equation~(\ref{eq:classical_action}) through the requirements of the Weyl group relations.
Since the results seem to be consistent with the prescription of the ``$q$-ordering" (or Weyl ordering) applied in~\cite{Moriyama}, it will be interesting to study whether and how such a prescription works in general.
\end{Remark}

\begin{Remark}
The quantum Weyl group actions on the subfield $\C(h_i,e_i,x,y)$ can be constructed from the quantum curves in~\cite{Moriyama} without difficulty.
In~\cite{Moriyama}, two realizations of the quantum curves, i.e., the ``triangular'' form and the ``rectangular'' form were constructed from a heuristic method by consulting previous classical results in~\cite{BBT,KY} and an empirical quantization rule in~\cite{KMN}.
The two realizations are related explicitly by a birational transformation, where each simple reflection $s_i$ is given by explicit actions on $\{h_i,e_i\}$, and besides, trivially on $\{x,y\}$ at least in one realization.
By composition, the nontrivial actions in one realization are transplanted from the trivial ones in the other and all the actions of $s_i$ in the subfield $\C(h_i,e_i,x,y)$ are obtained.
As~a~result, the actions of $s_i$ are identical to those anticipated from previous works by~\cite{Hasegawa} for~$W\big(D_5^{(1)}\big)$ and~\cite{KMN}\footnote{Note that it is necessary to generalize slightly from~\cite{KMN,Moriyama} to obtain the representation of the affine Weyl group by lifting the constraint on the parameters, since only symmetries of the quantum curve (which is non-affine) were discussed there.} for~$W\big(D_5^{(1)}\big)$, $W\big(E_7^{(1)}\big)$.
We emphasize that here the quantum Weyl group actions on the tau variables are also obtained.
Namely, inspired by the work~\cite{Kuroki2}, we have further noticed that the representations can be lifted by including the variables $\{\sigma_i, \tau_i\}$ as in equation~(\ref{eq:classical_action}).
Since the final result is quite simple and almost identical to the known classical case, we decide to take a quick style of presentation omitting the roundabout derivations.
With the quantum Weyl group actions on the tau variables identified, we can rederive the quantum curves from solid arguments.
\end{Remark}

\begin{Example}
For $w=s_{3,2,1,0,2,4,3}$, $e^\lambda:=w(e_1)=\frac{h_1^2 h_2}{e_1 e_7 e_9 e_{10} e_{11}}$, we have
\begin{gather}\label{eq:ex1_qu}
F_{\lambda}(x,y)=
\bigg(1+\frac{e_1 e_7 e_9 e_{10} e_{11}}{h_1^2 h_2} x\bigg) \bigg(1+\frac{q^{-1}}{e_1}x\bigg)
+e_{11}\bigg(1+\frac{e_7}{h_1} x\bigg) \bigg(1+\frac{e_9}{h_1} x\bigg)y.
\end{gather}
For $w=s_{0, 3, 4, 0, 2, 3, 2, 1, 0, 2, 4, 3}$, $e^\lambda:=w(e_{11})=\frac{h_1^2 h_2^2}{e_1 e_2 e_7 e_8 e_{10}^2 e_{11}}$, we have
\begin{align}
F_{\lambda}(x,y)={}&
\frac{x^2 \big(1+\frac{h_2}{e_{10}} y\big)\big(1+\frac{q h_2}{e_{10}} y\big)}{e_1 e_2q^2}
+\frac{x}{q}\bigg(1+\frac{h_2}{e_{10}} y\bigg) \bigg\{\bigg(\frac{1}{e_7}+\frac{1}{e_8}\bigg) \frac{h_1 h_2}{e_1e_2e_{10}} y+\bigg(\frac{1}{e_1}+\frac{1}{e_2}\bigg)\bigg\}\nonumber
\\
&+(1+e_{11} y) \bigg(1+\frac{h_1^2 h_2^2}{qe_1 e_2 e_7 e_8 e_{10}^2 e_{11}} y\bigg).
\label{eq:ex2_qu}
\end{align}
As expected, equations~(\ref{eq:ex1_qu}) and~(\ref{eq:ex2_qu}) reduce to equations~(\ref{eq:ex1_cl}) and~(\ref{eq:ex2_cl}) respectively when $q=1$.
\end{Example}

The representation can be realized as the adjoint actions as follows.
\begin{Theorem}\label{thm:adjoint}
The actions $s_i$ on variables $X=e_i, h_i, \tau_i, \sigma_i,x,y$ can be written as
\begin{gather*}
s_i(X)=G_i^{-1}r_i(X)G_i,
\\
G_0=\dfrac{\big(\frac{h_2}{e_{10}}y;q\big)^+_\infty}{(e_{11} y;q)^+_\infty}, \qquad
G_3=\dfrac{\big(\frac{1}{e_{1}}x;q\big)^+_\infty}{\big(\frac{e_7}{h_1} x;q\big)^+_\infty}, \qquad
G_i=1\qquad (i\neq 0,3),
\end{gather*}
where $(z;q)^+_{\infty}=\prod_{i=0}^{\infty}(1+q^i z)$ is the $q$-factorial and $r_i$ is a multiplicatively linear action on~$\{h_i, e_i, \sigma_i, \tau_i\}$ defined by $r_i(X)=s_i(X)|_{x=y=0}$, and $r_i(x)=x$, $r_i(y)=y$.
\end{Theorem}

\begin{proof}
Put $G=\dfrac{(\beta y;q)^+_\infty}{(\alpha y;q)^+_\infty}$.
By the relation $f(y)x=xf(qy)$ we have\footnote{We sometimes omit the base $q$ as $(z)^+_{\infty}=(z;q)^+_{\infty}$.
Note that our definition of the $q$-factorial is different from the conventional one $(z;q)_\infty=\prod_{i=0}^\infty(1-q^iz)$ by signs, which also appears later.}
\begin{gather*}
G^{-1}r_0(x) G=G^{-1}x G=\dfrac{(\alpha y)^+_\infty}{(\beta y)^+_\infty}x \dfrac{(\beta y)^+_\infty}{(\alpha y)^+_\infty}
=x\dfrac{(\alpha qy)^+_\infty}{(\beta qy)^+_\infty} \dfrac{(\beta y)^+_\infty}{(\alpha y)^+_\infty}
=x\dfrac{1+\beta y}{1+\alpha y}.
\end{gather*}
This gives the action $s_0(x)$ when $\alpha=e_{11}$, $\beta=\frac{h_2}{e_{10}}$, i.e., $G=G_0$.
Fortunately, the formula $G_0^{-1}r_0(*)G_0$ recovers the correct transformation for the other variables as well.
For instance
\begin{gather*}
G_0^{-1}r_0(\tau_{10})G_0=G_0^{-1}\frac{\sigma_2}{\tau_{11}}G_0=
G_0^{-1}\Big(G_0\big{|}_{\begin{subarray}{l}h_1\to q h_1,\\ e_{11}\to q e_{11}\end{subarray}}\Big)\frac{\sigma_2}{\tau_{11}}=(1+e_{11}y)\frac{\sigma_2}{\tau_{11}}.
\end{gather*}
The case $i=3$ is similar and the other cases are obvious.
\end{proof}

\begin{Remark}\label{remark6}
Using the realization $s_i$ in Theorem \ref{thm:adjoint}, one can give another proof of the Weyl group relations as follows.
We consider the most non-trivial case $s_0s_3s_0=s_3s_0s_3$ as an example.
Since
\begin{gather*}
s_0(X)=G_0^{-1}r_0(X)G_0,
\\[.5ex]
s_3s_0(X)=G_3^{-1}\big(r_3G_0^{-1}\big)(r_3r_0X)(r_3G_0)G_3,
\\[.5ex]
s_0s_3s_0(X)=G_0^{-1}\big(r_0G_3^{-1}\big)\big(r_0r_3G_0^{-1}\big)(r_0r_3r_0X)(r_0r_3G_0)(r_0G_3)G_0,
\end{gather*}
we have $s_0s_3s_0(X)=G^{-1}(r_0r_3r_0X)G$, where
\begin{gather*}
G=(r_0r_3G_0)(r_0G_3)G_0=\dfrac{\big(\frac{h_1h_2}{e_1e_7e_{10}}y\big)^+_\infty} {\big(\frac{h_2}{e_{10}}y\big)^+_\infty}
\dfrac{\big(\frac{1}{e_1}x\big)^+_\infty}{\big(\frac{e_7e_{10}e_{11}}{h_1h_2}x\big)^+_\infty}
\dfrac{\big(\frac{h_2}{e_{10}}y\big)^+_\infty}{\big(e_{11}y\big)^+_\infty}.
\end{gather*}
Similarly we have $s_3s_0s_3(X)={\tilde G}^{-1}(r_3r_0r_3X){\tilde G}$, where
\begin{gather*}
{\tilde G}=(r_3r_0G_3)(r_3G_0)G_3
=\dfrac{\big(\frac{e_7}{h_1}x\big)^+_\infty}{\big(\frac{e_7e_{10}e_{11}}{h_1h_2}x\big)^+_\infty}
\dfrac{\big(\frac{h_1h_2}{e_1e_7e_{10}}y\big)^+_\infty}{(e_{11}y)^+_\infty}
\dfrac{\big(\frac{1}{e_1}x\big)^+_\infty}{\big(\frac{e_7}{h_1}x\big)^+_\infty}.
\end{gather*}
Due to the relation $r_0r_3r_0=r_3r_0r_3$, the relation $s_0s_3s_0=s_3s_0s_3$ is guaranteed if $G={\tilde G}$.
Rescaling $y \to \frac{e_{10}}{h_2}y$, $x \to \frac{h_1}{e_7}x$ and putting $a=\frac{h_1}{e_1e_7}$, $b=\frac{e_{10}e_{11}}{h_2}$, the relation $G={\tilde G}$ reduces to the following identity which may be considered as a version of the quantum dilogarithm identity (see, e.g.,~\cite{Kirillov} and references therein).
\end{Remark}

\begin{Lemma}
For non-commuting variables $yx=q xy$, we have
\begin{gather}\label{eq:prod-id}
\dfrac{(ay)^+_\infty}{(y)^+_\infty}
\dfrac{(ax)^+_\infty}{(bx)^+_\infty}
\dfrac{(y)^+_\infty}{(by)^+_\infty}=
\dfrac{(x)^+_\infty}{(bx)^+_\infty}
\dfrac{(ay)^+_\infty}{(by)^+_\infty}
\dfrac{(ax)^+_\infty}{(x)^+_\infty}.
\end{gather}
\end{Lemma}

\begin{proof} By replacements $x \to -x$ and $y\to -y$, equation~(\ref{eq:prod-id}) can be written as
\begin{gather}\label{eq:prod-id-m}
\dfrac{(ay)_\infty}{(y)_\infty}
\dfrac{(ax)_\infty}{(bx)_\infty}
\dfrac{(y)_\infty}{(by)_\infty}=
\dfrac{(x)_\infty}{(bx)_\infty}
\dfrac{(ay)_\infty}{(by)_\infty}
\dfrac{(ax)_\infty}{(x)_\infty},
\end{gather}
where $(x)_\infty=\prod_{i=0}^{\infty}(1-q^i x)$, and we will prove equation~(\ref{eq:prod-id}) in this form.
We recall the $q$-binomial identity
\begin{gather}\label{eq:binom}
\dfrac{(az)_\infty}{(z)_\infty}
=\sum_{n\geq 0}{\dfrac{(a)_n}{(q)_n}z^n},\qquad
(a)_n=\frac{(a)_\infty}{(aq^n)_\infty},
\end{gather}
which follows by solving the difference equation $f(q z)=\frac{1-z}{1-az}f(z)$ for $f(z)=\frac{(az)_\infty}{(z)_\infty}$ in series expansion.
Using equation~(\ref{eq:binom}) and $yx=qxy$, the factors in equation~(\ref{eq:prod-id-m}) can be reordered as
\begin{gather*}
\dfrac{(ay)_\infty}{(y)_\infty}\dfrac{(ax)_\infty}{(bx)_\infty}
=\sum_{n\geq 0}\dfrac{(a)_n}{(q)_n}y^n\dfrac{(ax)_\infty}{(bx)_\infty}
=\sum_{n\geq 0}\dfrac{(a)_n}{(q)_n}\dfrac{(aq^nx)_\infty}{(bq^nx)_\infty}y^n
=\dfrac{(ax)_\infty}{(bx)_\infty}\sum_{n\geq 0}\dfrac{(a)_n}{(q)_n}\dfrac{(bx)_n}{(ax)_n}y^n,
\\
\dfrac{(ay)_\infty}{(by)_\infty}\dfrac{(ax)_\infty}{(x)_\infty}
=\dfrac{(ay)_\infty}{(by)_\infty}\sum_{n\geq 0}\dfrac{(a)_n}{(q)_n}x^n
=\sum_{n\geq 0}x^n\dfrac{(a)_n}{(q)_n}\dfrac{(aq^ny)_\infty}{(bq^ny)_\infty}
=\sum_{n\geq 0}x^n\dfrac{(a)_n}{(q)_n}\dfrac{(by)_n}{(ay)_n}\dfrac{(ay)_\infty}{(by)_\infty}.
\end{gather*}
Hence, equation~(\ref{eq:prod-id-m}) can be written as
\begin{gather}\label{eq:red-PPP}
(ax)_\infty\sum_{n\geq 0}\dfrac{(a)_n}{(q)_n}\dfrac{(bx)_n}{(ax)_n}y^n (y)_\infty
=(x)_\infty\sum_{n\geq 0}x^n\dfrac{(a)_n}{(q)_n}\dfrac{(by)_n}{(ay)_n}(ay)_\infty.
\end{gather}

Since the both hand sides of equation~(\ref{eq:red-PPP}) are written in the same ordering in $x$, $y$, whether the equality holds or not is independent of the commutation relation of $x$, $y$.
We will show it in~the commutative case, where equation~(\ref{eq:red-PPP}) can be written as
\begin{gather}\label{eq:iterated-Heine}
(ax)_\infty\ {}_2 \varphi_{1}\Big(\!\begin{array}{cc}{a,bx}\\{ax}\end{array},y\Big)\ (y)_\infty
=(x)_\infty\ {}_2 \varphi_{1}\Big(\!\begin{array}{cc}{a,by}\\{ay}\end{array},x\Big)\ (ay)_\infty,
\end{gather}
using the Heine's $q$-hypergeometric series
\begin{gather*}
{}_2\varphi_{1}\Big(\!\begin{array}{cc}{a,b}\\{c}\end{array},x\Big)
=\sum_{n\geq 0}\dfrac{(a)_n(b)_n}{(q)_n(c)_n}x^n.
\end{gather*}
Then equation~(\ref{eq:iterated-Heine}) can be confirmed via iterative use of the Heine's identity and the trivial symmetry relation
\begin{gather*}
{}_2\varphi_{1}\Big(\!\begin{array}{cc}{a,b}\\{c}\end{array},x\Big)
=\dfrac{(ax)_\infty}{(x)_\infty}
\dfrac{(b)_\infty}{(c)_\infty}
{}_2\varphi_{1}\Big(\!\begin{array}{cc}{c/b,x}\\{ax}\end{array},b\Big),\qquad
{}_2\varphi_{1}\Big(\!\begin{array}{cc}{a,b}\\{c}\end{array},x\Big)=
{}_2\varphi_{1}\Big(\!\begin{array}{cc}{b,a}\\{c}\end{array},x\Big).
\end{gather*}
The former is also obtained from the $q$-binomial identity.
\end{proof}

\begin{Proposition}
We put $k_1$, $k_2$ as the same as the classical case \eqref{eq:k1k2},
\begin{gather*}
k_1=x \dfrac{\tau_{10}}{\tau_{11}}, \qquad
k_2=y \dfrac{\tau_7\tau_8\tau_9}{\tau_1\tau_2\cdots \tau_6}.
\end{gather*}
Then $k_1$, $k_2$ are $W\big(E^{(1)}_8\big)$ invariant also in the quantum setting.
\end{Proposition}

\begin{proof} We will check only the nontrivial actions and they go as
\begin{gather*}
s_0\bigg(x \dfrac{\tau_{10}}{\tau_{11}}\bigg)=x \dfrac{1+y\frac{h_2}{e_{10}}}{1+y e_{11}}(1+y e_{11})\dfrac{\sigma_2}{\tau_{11}}
\dfrac{1}{1+y \frac{h_2}{e_{10}} }\dfrac{\tau_{10}}{\sigma_2}=x \dfrac{\tau_{10}}{\tau_{11}},
\end{gather*}
and
\begin{gather*}
s_3\bigg(\dfrac{\tau_7}{\tau_1}y\dfrac{\tau_8\tau_9}{\tau_2\cdots\tau_6}\bigg)=
\dfrac{\tau_7}{\sigma_1}\dfrac{1}{1+x \frac{e_7}{h_1}}\dfrac{\sigma_1}{\tau_7}\bigg(1+\dfrac{x}{e_1}\bigg)\dfrac{1+x \frac{e_7}{h_1}}{1+\frac{x}{e_1}}y \dfrac{\tau_8\tau_9}{\tau_2\cdots\tau_6}=\dfrac{\tau_7}{\tau_1}y\dfrac{\tau_8\tau_9}{\tau_2\cdots\tau_6}.
 \tag*{\qed}
\end{gather*}
\renewcommand{\qed}{}
\end{proof}

Due to this proposition, the actions of $w \in W\big(E^{(1)}_8\big)$ on $x, y$ can be reduced to the actions on~$\sigma_i$, $\tau_i$ as in the classical case.

\section{Non-logarithmic property}\label{section3}

From the several examples of the quantum polynomials as in equations~(\ref{eq:ex1_qu}) and~(\ref{eq:ex2_qu}), one observes an interesting factorization in their coefficients, which was utilized in constructing quantum curves in~\cite{Moriyama}.
We will clarify the meaning of such factorizations from the viewpoint of the $q$-difference operators.

Consider a $q$-difference equation $D\psi(x)=0$, $D=\sum_{i=0}^{d_1}x^iA_i(y)$, $(yx=q xy)$.
We look for a~solution $\psi(x)$ around $x=0$ of the form
\begin{gather*}
\psi(x)=x^\rho\sum_{j=0}^{\infty}c_jx^j \qquad (c_0\neq 0).
\end{gather*}
From the coefficient of $x^{\rho+k}$ in the equation $D\psi(x)=0$, we have
\begin{gather*}
\sum_{i+j=k}A_i\big(q^{\rho+j}\big)c_j =A_k\big(q^{\rho}\big)c_0+A_{k-1}\big(q^{\rho+1}\big)c_1+\cdots+A_0\big(q^{\rho+k}\big)c_k=0,
\end{gather*}
where $A_i(y)=0$ for $i>d_1$.
The (multiplicative) exponents $y=q^\rho$ are determined as the zeros of $A_0(y)$.
Then the coefficients $c_1, c_2, \dots$ will be determined recursively.
For $c_k$, we have the following cases:
\begin{itemize}\itemsep=0pt
\item[$(1)$] If $A_0\big(q^{\rho+k}\big)\neq 0$, then $c_k$ is uniquely determined from $c_0, c_1, \dots, c_{k-1}$.
\item[$(2a)$] If $A_0\big(q^{\rho+k}\big)= 0$ and $X_k:=A_k\big(q^{\rho}\big)c_0+A_{k-1}\big(q^{\rho+1}\big)c_1+\cdots+A_1\big(q^{\rho+k-1}\big)c_{k-1}\neq 0$, then the equation for $c_k$ has no solution and we do not have the power series solution (one should consider a solution with logarithmic terms in $x$).
\item[$(2b)$] If $A_0\big(q^{\rho+k}\big)= 0$ and $X_k=0$, then the coefficient $c_k$ is free and we still have series solutions with exponents $y=q^\rho, q^{\rho+k}$.
\end{itemize}
For the last case $(2b)$, the difference operator $D$ admits a non-logarithmic solution around $x=0$ and $x=0$ is called ``non-logarithmic'' singularity of $D$.
Non-logarithmic singularities around $x=\infty$ (or $y=0$ or $y=\infty$) are defined similarly.
If we apply the condition of non-logarithmic singularities to the case with successive exponents, coefficients of the $q$-difference operator $D$ are constrained strongly by the non-logarithmic properties of its solution as follows.

\begin{Proposition} For a difference operator $D=\sum_{i=0}^{d_1}x^iA_i(y)$, we have
\begin{itemize}\itemsep=0pt
\item[$(1)$] $D$ has non-logarithmic singularities at $x=0$ with $y=a, qa, \dots, q^{m-1}a
\Leftrightarrow A_i(y) \propto \prod_{j=0}^{m-i-1}(y-q^ja)$ for $0\leq i\leq m-1$,
\item[$(2)$] $D$ has non-logarithmic singularities at $x=\infty$ with $y=a,q^{-1}a, \dots, q^{-m+1}a
\Leftrightarrow A_i(y) \propto \prod_{j=0}^{m-i-1}(y-q^{-j}a)$ for $d_1-m+1\leq i\leq d_1$.
\end{itemize}
Similarly, for a difference operator $D=\sum_{i=0}^{d_2}B_i(x)y^i$, we have
\begin{itemize}\itemsep=0pt
\item[$(3)$] $D$ has non-logarithmic singularities at $y=0$ with $x=a, qa, \dots, q^{m-1}a
\Leftrightarrow B_i(x) \propto \prod_{j=0}^{m-i-1}(x-q^ja)$ for $1\leq i\leq m$,
\item[$(4)$] $D$ has non-logarithmic singularities at $y=\infty$ with $x=a, q^{-1}a, \dots, q^{-m+1}a
\Leftrightarrow B_i(x) \propto \prod_{j=0}^{m-i-1}(x-q^{-j}a)$ for $d_2-m+1\leq i\leq d_2$.
\end{itemize}
\end{Proposition}

\begin{proof}
Consider the case (1) (the other cases are similar).
For the non-logarithmic property with successive exponents, the recursion relations for the power series solution
{\samepage\begin{gather*}
A_0(y)c_0=0,
\\[1ex]
A_1(y)c_0+A_0(qy)c_1=0,
\\[1ex]
\cdots\cdots\cdots\cdots\cdots\cdots\cdots\cdots
\\[1ex]
A_{m-1}(y)c_0+\cdots +A_0\big(q^{m-1}y\big)c_{m-1}=0,
\end{gather*}}\noindent
should be satisfied termwise with $m$ free coefficients: $c_0, \dots, c_{m-1}$.
From the first relation we have $A_0(y)\propto \prod_{j=0}^{m-1}(y-q^ja)$, and the other factorizations also follow easily.
\end{proof}

In other words, a $q$-difference operator $D=\sum_{i=0}^{d_1}x^iA_i(y)$ with boundary coefficients $A_0(y)$, $A_{d_2}(y)$ having zeros successive in powers of $q$, is non-logarithmic iff suitable parts of the zeros penetrate into the internal coefficients.
We have similar properties for a difference operator $D=\sum_{i=0}^{d_2}B_i(x)y^i$ also.
The non-logarithmic property of $q$-difference operators plays important roles in the following characterization of quantum polynomials and also in~\cite{NagY,NRY,Take,Yamada17} etc.

\section[The F-polynomials]{The $\boldsymbol F$-polynomials}\label{section4}

Here we study the quantum analog of the polynomials $F_\lambda(x,y)$ in Proposition \ref{phi_lambda}.

\begin{Definition}
For each degree/multiplicity data $\lambda=((d_1,d_2), (m_1,\dots,m_{11})) \in P$, we define a non-commutative polynomial $F=F_{\lambda}(x,y)=F_{\lambda}(x,y; \{h_i,e_i\})$
by the following conditions:
\begin{itemize}\itemsep=0pt
\item[$(x)_{\lambda}$] Collecting terms with the same power of $x$, the polynomial $F$ takes the form
\begin{gather*}
F=
\sum_{i=0}^{d_1} x^i
\prod_{t=i}^{m_{11}-1}\big(1+q^te_{11}y\big)
\prod_{t=d_1-m_{10}}^{i-1}\bigg(1+q^t\frac{h_2}{e_{10}}y\bigg)
U_i(y),
\end{gather*}
where $U_i(y)$ is a polynomial\footnote{If there appear many polynomials $U_i(y)$ of the same degree, they should be considered as different ones.
This applies to $V_i(x)$ in equation~(\ref{eq:y-cond}) as well.} in $y$ of degree $d_2-(i-d_1+m_{10})_{+}-(m_{11}-i)_{+}$.

\item[$(y)_{\lambda}$] Collecting terms with the same power of $y$, the polynomial $F$ takes the form
\begin{gather}\label{eq:y-cond}
F=\sum_{i=0}^{d_2}
\prod_{k=1}^6\prod_{t=i-m_k}^{-1}\bigg(1+q^t\frac{1}{e_k}x\bigg)
\prod_{k=7}^9\prod_{t=0}^{i-d_2+m_k-1}\bigg(1+q^t\frac{e_k}{h_1}x\bigg)
V_i(x)\, y^i,
\end{gather}
where $V_i(x)$ is a polynomial in $x$ of degree $d_1-\sum_{k=1}^6(m_{k}-i)_{+}-\sum_{k=7}^9(i-d_2+m_{k})_{+}$.
\end{itemize}
In these conditions, $(x)_{+}=\max(x,0)$ and the empty product is $1$: $\prod_{t=a}^b(*)=1$ $(a>b)$.
\end{Definition}

\begin{Remark}
For the $q=1$ case, it is easy to see that the conditions $(x)_{\lambda}$, $(y)_\lambda$ reduce to the conditions specified by the degree/multiplicity data $\lambda=(d_i,m_i)$.
Hence the quantum polynomial $F_\lambda(x,y)$ reduces to the classical polynomial $F_\lambda(x,y)$ in Proposition \ref{phi_lambda}.
\end{Remark}

\begin{Proposition}\label{e8weylorbit}
Let $\Lambda$ be the $W\big(E^{(1)}_8\big)$-orbit of $\{E_1, \dots, E_{11}\}$.
Then for $\lambda \in \Lambda$, the polyno\-mial~$F_\lambda(x,y)$ exists and is unique up to a normalization.
We will normalize it by $F_\lambda(0,0)=1$.
\end{Proposition}

\begin{proof}
The conditions $(x)_\lambda$, $(y)_\lambda$ give linear equations\footnote{In the commutative case, this is known as the linear system $|\lambda|$.} (vanishing conditions) for $F_\lambda(x,y)$.
Counting the numbers of coefficients and equations, the dimension of the solution is given by
\begin{gather}\label{eq:dim}
{\rm dim}=(d_1+1)(d_2+1)-\sum_{k=1}^{11} \frac{m_k(m_k+1)}{2}
=\dfrac{1}{2}\lambda\cdot\lambda+\dfrac{1}{2}\lambda\cdot \delta_{\rm Red}+1,
\end{gather}
where $\lambda$ is in equation~(\ref{eq:add_lambda}), dot($\cdot$) is the intersection pairing and $\delta_{\rm Red}=2 H_1+2H_2-\sum_{i=1}^{11} E_i$.
Then, for $\lambda \in \Lambda$ we have ${\rm dim}=1$, since $\lambda\cdot \lambda=-1$ and $\lambda\cdot \delta_{\rm Red}=1$.
\end{proof}

We use a notation $s^*_i$ to represent the induced action on the data $\lambda=(d_i,m_i)$ defined by~$s_i\big(e^{\lambda}\big)=e^{s^*_i \lambda}$, hence $s_i\big(\tau^\lambda\big)\vert_{x=y=0}=\tau^{s^*_i \lambda}$.
It is explicitly given as
\begin{gather*}
s^*_0=\{d_2\mapsto d_1+d_2-m_{10}-m_{11}, \,
{m}_{10}\mapsto d_1-m_{11}, \,
{m}_{11}\mapsto d_1-m_{10}\},
\\[.5ex]
s^*_1=\{m_8 \leftrightarrow m_9\}, \qquad
s^*_2=\{m_7 \leftrightarrow m_8\},
\\[.5ex]
s^*_3=\{d_1\mapsto d_1+d_2-m_{1}-m_{7}, \,
{m}_{1}\mapsto d_2-m_{7}, \,
{m}_{7}\mapsto d_2-m_{1}\}, \qquad
s^*_4=\{m_1 \leftrightarrow m_2\},
\\[.5ex]
s^*_5=\{m_2 \leftrightarrow m_3\}, \qquad
s^*_6=\{m_3 \leftrightarrow m_4\}, \qquad
s^*_7=\{m_4 \leftrightarrow m_5\}, \qquad
s^*_8=\{m_5 \leftrightarrow m_6\}.
\end{gather*}
The following is the main result of this paper.
\begin{Theorem}\label{thm:sF}
Let $F_{\lambda}(x,y)$ be a polynomial satisfying the conditions $(x)_\lambda$, $(y)_\lambda$.
Then for each simple reflection $s_i \in W\big(E_8^{(1)}\big)$, the function $F_{s^*_i{\lambda}}(x,y)$ defined by
\begin{gather}\label{eq:sFaction}
s_i\big(F_\lambda(x,y) \tau^\lambda\big)=F_{s^*_i{\lambda}}(x,y)\tau^{s^*_i{\lambda}}, \qquad
\tau^{\lambda}=\dfrac{\sigma_1^{d_1}\sigma_2^{d_2}}{\tau_1^{m_1}\cdots \tau_{11}^{m_{11}}},
\end{gather}
is also a polynomial in $x$, $y$ and satisfy the condition $(x)_{s^*_i{\lambda}}$, $(y)_{s^*_i{\lambda}}$.
In particular, for $\lambda \in \Lambda$, the unique normalized polynomials $F_\lambda(x,y)$ can be obtained by the actions~\eqref{eq:sFaction} from the initial condition $F_{e_i}=1$.
\end{Theorem}

\begin{Remark}
The polynomial $F_{\lambda}(x,y)$ is not a function but a section of a line bundle $\mathcal{L}_{\lambda}$ on~$X$, and equation~(\ref{eq:sFaction}) can be considered as its trivialization in the commutative case~\cite{NY:cmp,NY:birat}.
Theorem~\ref{thm:sF} suggests a non-commutative analog of such a geometric understanding.
\end{Remark}

\begin{Example}
For $e^\lambda=\frac{h_1h_2}{e_{10}e_{11}}$, the corresponding $F_\lambda$ has two parameters:
\begin{gather}\label{eq:ex-two-para}
\begin{array}{l}
F_{\lambda}
=c_0 (1+e_{11} y)+c_1 x\bigg(1+\dfrac{h_2}{e_{10}}y\bigg).
\end{array}
\end{gather}
Then, we have
\[
s_3\bigg(F_\lambda\frac{\sigma_1\sigma_2}{\tau_{10}\tau_{11}}\bigg)=
\tilde{F}_{\tilde\lambda}\frac{\sigma_1^2\sigma_2}{\tau_1\tau_7\tau_{10}\tau_{11}},
\]
where
\begin{align*}
\tilde{F}_{\tilde\lambda}&=(c_0+c_1x)\bigg(1+\dfrac{1}{q e_1}x\bigg)+\bigg(1+\dfrac{e_7}{h_1}x\bigg)\bigg(c_0e_{11}+c_1 \dfrac{h_1h_2}{e_1e_7e_{10}}x\bigg)y
\\
&=c_0(1+e_{11}y)+x\bigg\{c_0\bigg(\dfrac{1}{q e_1}+\dfrac{e_7e_{11}}{h_1}y\bigg)+c_1\bigg(1+\dfrac{h_1h_2}{e_1e_7e_{10}}y\bigg)\bigg\}+c_1\dfrac{1}{q e_1}x^2\bigg(1+q \dfrac{h_2}{e_{10}}y\bigg).
\end{align*}
We see that the polynomial $\tilde{F}_{\tilde\lambda}$ gives a general solution for the condition $(x)_{\tilde{\lambda}}$, $(y)_{\tilde{\lambda}}$, where $e^{\tilde{\lambda}}=s_3\big(e^\lambda\big)=\frac{h_1^2h_2}{e_1e_7e_{10}e_{11}}$.
\end{Example}

\begin{proof}[Proof of Theorem \ref{thm:sF}.]
We will consider the cases $s_0$ and $s_3$ (other cases are obvious).

\medskip\noindent
{\it Case $s_0$}.
Let $F=F_\lambda(x,y)$ be a polynomial satisfying the condition $(x)_\lambda$.
We compute the action of $s_0$ on $F \tau^{\lambda}$.
For $F$, we have
\begin{align*}
F&=\sum_{i=0}^{d_1}x^i\prod_{t=i}^{m_{11}-1}\big(1+q^te_{11}y\big)
\prod_{t=d_1-m_{10}}^{i-1}\bigg(1+q^t\frac{h_2}{e_{10}}y\bigg)U_i(y)
\\
&\underset{s_0}{\mapsto}
\sum_{i=0}^{d_1} x^i \prod_{t=0}^{i-1}\dfrac{1+q^t\frac{h_2}{e_{10}} y}{1+q^t e_{11}y}
\prod_{t=i}^{m_{11}-1}\bigg(1+q^t\frac{h_2}{e_{10}}y\bigg)
\prod_{t=d_1-m_{10}}^{i-1}\big(1+q^te_{11}y\big) \tilde{U}_i(y),
\end{align*}
where $\tilde{U}_i(y)$ is a polynomial in $y$ of degree $i$.
For $\tau^\lambda$, considering only the relevant factors, we~have
\begin{align*}
\dfrac{\sigma_1^{d_1}\sigma_2^{d_2}}{\tau_{10}^{m_{10}}\tau_{11}^{m_{11}}} &=
\tau_{11}^{-m_{11}}\tau_{10}^{d_1-m_{10}}\bigg(\dfrac{\sigma_1}{\tau_{10}}\bigg)^{d_1}\sigma_2^{d_2}
\\
& \underset{s_0}{\mapsto}
\prod_{t=0}^{m_{11}-1}\frac{1}{1+q^t\frac{h_2}{e_{10}}y}
\prod_{t=0}^{d_1-m_{10}-1}(1+q^te_{11} y)
\dfrac{\sigma_1^{d_1}\sigma_2^{d_1+d_2-m_{10}-m_{11}}}{\tau_{10}^{d_1-m_{11}}\tau_{11}^{d_1-m_{10}}}.
\end{align*}
Collecting the factors $\big(1+q^t e_{11}y\big)$ and $\big(1+q^t\frac{h_2}{e_{10}}y\big)$, we have $s_0(F \tau^{\lambda})=\tilde{F}\tau^{s_0 \lambda}$, where
\begin{gather*}
\tilde{F}=\sum_{i=0}^{d_1}x^i\prod_{t=i}^{d_1-m_{10}-1}\big(1+q^te_{11}y\big)
\prod_{t=m_{11}}^{i-1}\bigg(1+q^t\frac{h_2}{e_{10}}y\bigg)\tilde{U}_i(y).
\end{gather*}
Note that here we have applied the formula $\prod_{t=u}^{v-1}(*)\big[\prod_{t=w}^{u-1}(*)\big/\prod_{t=w}^{v-1}(*)\big]=\prod_{t=v}^{u-1}(*)$, which holds for $w\le\min(u,v)$.
Hence, $\tilde{F}$ is a polynomial of bidegree $(d_1, \tilde{d}_2=d_1+d_2-m_{10}-m_{11})$ satisfying the condition $(x)_{\tilde{\lambda}}$ for $\tilde{\lambda}=s_0(\lambda)$.
Moreover, $\tilde{F}$ satisfies the condition $(y)_{\tilde{\lambda}}$ also.
To~confirm this, we note that the condition $(y)_\lambda$ is equivalent to the condition on the top and bottom coefficients of $F=\sum_{i=0}^{d_2}A_i(x)y^i$:
\begin{gather*}
A_0 =\mathop{\rm const} \prod_{k=1}^6\prod_{t=-m_k}^{-1}\bigg(1+q^t\frac{1}{e_k}x\bigg), \qquad
A_{d_2}=\mathop{\rm const} \prod_{k=7}^9\prod_{t=0}^{m_k-1}\bigg(1+q^t\frac{e_k}{h_1}x\bigg),
\end{gather*}
together with the non-logarithmic properties.
For the coefficients $\tilde{A}_0$, $\tilde{A}_{{\tilde d}_2}$ of $\tilde{F}$, we have obviously $\tilde{A}_0=s_0(A_0)=\mathop{\rm const}A_0$, and we also have
\begin{gather*}
\tilde{A}_{{\tilde d}_2}
=s_0(A_{d_2})=\mathop{\rm const} \prod_{k=7}^9\prod_{t=0}^{m_k-1}\bigg(1+q^t\frac{e_k}{s_0(h_1)}\frac{h_2}{e_{10}e_{11}}x\bigg)
=\mathop{\rm const} A_{d_2},
\end{gather*}
since $s_0(x)=x \frac{1+\frac{h_2}{e_{10}}y}{1+e_{11}y} \to \frac{h_2}{e_{10}e_{11}}x$ ($y \to \infty$).
Hence, the leading coefficients $\tilde{A}_0$, $\tilde{A}_{{\tilde d}_2}$ have the required from $(y)_{\tilde{\lambda}}$.
Our remaining task is to show that the non-logarithmic property of $\tilde{F}$ is inherited from that of $F$.
Indeed, recall that the $s_0$-transformation is realized as the adjoint action $s_0(X)=G_0^{-1} r_0(X) G_0$ with $G_0=\big(y \frac{h_2}{e_{10}}\big)^+_\infty/(y e_{11})^+_\infty$.
Then, under the corresponding transformation of the solutions $\psi(y) \underset{s_0}{\mapsto} G_0(y)^{-1} r_0(\psi(y))$, the non-logarithmic property around $y=0$ is preserved from the regularity of the $q$-factorial $(z)^+_\infty$.
Besides, with the rewriting
\begin{gather*}
G_0=\dfrac{\big(y \frac{h_2}{e_{10}}\big)^+_\infty}{(y e_{11})^+_\infty}
=C(y) y^{\nu}\dfrac{(q/(y e_{11}))^+_\infty}{\big(q/\big(y\frac{h_2}{e_{10}}\big)\big)^+_\infty}, \qquad
q^\nu=\dfrac{e_{10}e_{11}}{h_2},
\end{gather*}
we can use $\tilde{G}_0=y^{\nu}{(q/(y e_{11}))^+_\infty}/{\big(q/\big(y\frac{h_2}{e_{10}}\big)\big)^+_\infty}$ instead of $G_0$, since the factor $C(y)$ is a pseudo constant: $C(qy)=C(y)$ and irrelevant for the adjoint action.
Hence the non-logarithmic property around $y=\infty$ follows similarly.

\medskip\noindent
{\it Case $s_3$}.
Let $F=F_\lambda(x,y)$ be a polynomial satisfying the condition $(y)_\lambda$.
The action of $s_3$ on two parts of $F\tau^\lambda$ is given by
\begin{gather*}
\prod_{t=i-m_1}^{-1}\bigg(1+q^t\dfrac{1}{e_1}x\bigg)
\prod_{t=0}^{i-d_2+m_7-1}\bigg(1+q^t\dfrac{e_7}{h_1}x\bigg)
V_i(x)\, y^i
\\ \qquad
\underset{s_3}{\mapsto}
\prod_{t=i-m_1}^{-1}\bigg(1+q^t\dfrac{e_7}{h_1}x\bigg)
\prod_{t=0}^{i-d_2+m_7-1}\bigg(1+q^t\dfrac{1}{e_1}x\bigg)
\tilde{V}_i(x) \prod_{t=1}^{i-1}\dfrac{1+q^t \frac{e_7}{h_1}x}{1+q^t\frac{1}{e_1}x} y^i,
\end{gather*}
and
\begin{align*}
\dfrac{\sigma_1^{d_1}\sigma_2^{d_2}}{\tau_{1}^{m_{1}}\tau_{7}^{m_{7}}} =
\tau_{1}^{-m_{1}}\tau_{7}^{d_2-m_{7}}\bigg(\dfrac{\sigma_2}{\tau_{7}}\bigg)^{d_2}\sigma_1^{d_1}
\underset{s_3}{\mapsto}y^{-i}
\prod_{t=i-m_1}^{i-1}\frac{1}{1+q^t \frac{e_7}{h_1}x}
\prod_{t=i-d_2+m_7}^{i-1}\bigg(1+q^t \frac{1}{e_1} x\bigg)y^i,
\end{align*}
where we choose $\tilde\lambda=s^*_3\lambda$ with the tilde applying to each component of $\lambda=(d_i,m_i)$.
By~combining them, the factors $\big(1+q^t\frac{1}{e_1}x\big)$ and $\big(1+q^t\frac{e_7}{h_1}x\big)$ in the coefficient of $y^i$ in $s_3\big(F \tau^{\lambda}\big)$ are
\begin{gather*}
\prod_{t=i-d_2+m_7}^{-1}\bigg(1+q^t\dfrac{1}{e_1}x\bigg)
\prod_{t=0}^{i-m_1-1}\bigg(1+q^t\dfrac{e_7}{h_1}x\bigg),
\end{gather*}
where we have applied the formula $\prod_{t=u}^{v-1}(*)\big[\prod_{t=v}^{w-1}(*)\big/\prod_{t=u}^{w-1}(*)\big]=\prod_{t=v}^{u-1}(*)$, which holds for $\max(u,v)\le w$.
Then, we have $s_3\big(F\tau^{\lambda}\big)=\tilde{F}\tau^{\tilde{\lambda}}$ with
\begin{gather*}
\tilde{F}=
\sum_{i=0}^{d_2}
\prod_{k=1}^6\prod_{t=i-\tilde{m}_k}^{-1}\bigg(1+q^t\frac{1}{e_k}x\bigg)
\prod_{k=7}^9\prod_{t=0}^{\tilde{m}_k-\tilde{d}_2+i-1}\bigg(1+q^t\frac{e_k}{h_1}x\bigg)
\tilde{V}_i(x)\, y^i,
\end{gather*}
and hence, $\tilde{F}$ satisfies the condition $(y)_{\tilde{\lambda}}$ as desired.
The condition $(x)_{\tilde{\lambda}}$ can be confirmed using the adjoint action realization of $s_3$.
\end{proof}

\section[Quantum E8 curve as the Weyl group invariant]
{Quantum $\boldsymbol{E_8}$ curve as the Weyl group invariant}\label{section5}

Consider a degree/multiplicity data
\begin{gather}\label{eq:63-123}
\lambda=(d_i,m_i)=((6,3),(1,1,1,1,1,1,2,2,2,3,3)),
\end{gather}
which is special since it is invariant $w(\lambda)=\lambda$ under $w \in W\big(E_8^{(1)}\big)$, i.e., $\{\lambda\}$ is a Weyl orbit with only a single element, hence $\lambda$ is not in the Weyl orbit $\Lambda$ of Proposition \ref{e8weylorbit}.
We look for the corresponding quantum polynomial $P=P(x,y)$ defined by the conditions $(x)_\lambda$, $(y)_\lambda$:
\begin{align}
P={}&y^{[0]}\prod_{t=0}^{2}\big(1+q^te_{11}y\big)+x y^{[1]}\prod_{t=1}^{2}\big(1+q^te_{11}y\big)
+x^2 y^{[2]} \big(1+q^2e_{11} y\big)+x^3 y^{[3]}\nonumber
\\
&+x^4 y^{[2]} \bigg(1+q^3\frac{h_2}{e_{10}}y\bigg)
+x^5 y^{[1]} \prod_{t=3}^{4}\bigg(1+q^t\frac{h_2}{e_{10}}y\bigg)
+x^6 y^{[0]} \prod_{t=3}^{5}\bigg(1+q^t\frac{h_2}{e_{10}}y\bigg)\nonumber
\\
={}&x^{[0]} \prod_{k=1}^6\! \bigg(1\!+\frac{1}{q e_k} x\bigg)
\!+x^{[6]} y
\!+x^{[3]} \prod_{k=7}^9\! \bigg(1\!+\frac{e_k}{h_1}x\bigg) y^2
\!+x^{[0]} \prod_{k=7}^9\! \prod_{t=0}^{1}\bigg(1\!+q^t\frac{e_k}{h_1}x\bigg) y^3,
\label{eq:Hcond}
\end{align}
where $x^{[i]}$ \big(or $y^{[i]}$\big) represent some polynomials in $x$ (or $y$) of degree $i$.

\begin{Proposition}
When the parameters satisfy the constraint
\begin{gather}\label{eq:ehcons}
h_1^6h_2^3=e_1e_2e_3e_4e_5e_6e_7^2e_8^2e_9^2e_{10}^3e_{11}^3,
\end{gather}
then the general solution $P(x,y)$ of the condition~\eqref{eq:Hcond} takes the form
\begin{gather}\label{eq:PP0_form}
P(x,y)=c_0 P_0(x,y)+c_1 x^3y.
\end{gather}
Moreover, when the polynomial $P(x,y)$ is normalized as $P(0,0)=1$ and $c_1 \in \C$, then $P(x,y)$ is invariant under the action of $W\big(E_8^{(1)}\big)$ up to some multiplicative factors, namely
\begin{gather}\label{eq:Hinv}
s_i(P)=P\qquad (i\neq 0,3), \qquad
s_0(P)=P\prod_{i=0}^{2}\dfrac{1+q^i\frac{h_2}{e_{10}} y}{1+q^i e_{11}y}, \qquad
s_3(P)=P\dfrac{1+\frac{e_7}{q h_1}x}{1+\frac{1}{q e_1}x}.
\end{gather}
\end{Proposition}
\begin{proof} Consider an auxiliary case
\begin{gather*}
\mu=((d_1,d_2), (m_1, \dots, m_{11}))=((6,3),(0,1,1,1,1,1,2,2,2,3,3)).
\end{gather*}
From equation~(\ref{eq:dim}), the general solution $F_\mu(x,y)$ for the condition $(x)_{\mu}$, $(y)_{\mu}$ has two linearly independent solutions.
We can and we will choose a basis $\big\{P_0(x,y), x^3 y\big\}$, where $P_0(x,y)$ is fixed by the conditions: $(i)$ the coefficient of $x^3 y$ in $P_0(x,y)$ is zero, and $(ii)$ $P_0(0,0)=1$.
By~defini\-tion $F_\mu(x,0)$ has 6 roots at $x=a, e_2, \dots, e_6$, where $a$ is determined by
$h_1^6h_2^3=ae_2e_3e_4e_5e_6e_7^2e_8^2e_9^2e_{10}^3e_{11}^3$ due to the relation between the roots and the coefficients.
Now we turn to the case $\lambda$ in~equa\-tion~(\ref{eq:63-123}).
Compared with the case $\mu$, the case $\lambda$ demands one more condition $P(e_1,0)=0$.
However this extra condition is automatically satis\-fied if $a=e_1$, i.e., the constraint (\ref{eq:ehcons}) is satis\-fied.
Hence, under the constraint (\ref{eq:ehcons}), the general solution $F_\lambda(x,y)$ is given by $F_\mu(x,y)|_{a=e_1}$ which has the desired form (\ref{eq:PP0_form}).
Equation~(\ref{eq:Hinv}) follows from Theorem~\ref{thm:sF} and explicit computation on the monomial $x^3y$.
\end{proof}

\begin{Corollary}
The quotient $H(x,y)=P(x,y)x^{-3}y^{-1}$ is invariant under $W\big(E_8^{(1)}\big)$.
\end{Corollary}

The quantum curve $\widehat{H}(Q,P; \{f_i,g_i,h_i\})$ for $E_8$ in~\cite{Moriyama} written in the ``rectangular'' realization coincides with $H(x,y; \{h_i, e_i\})=x^{-3}P_0\big(x,q^{-3}y\big)y^{-1}$ up to a normalization, by the following change of the variables and parameters,\footnote{Note that the symbols $h_i$ have different meanings in $\widehat{H}(Q,P; \{f_i,g_i,h_i\})$~\cite{Moriyama} and $H(x,y; \{h_i, e_i\})$ here.}
\begin{gather*}
(Q,P)\to q^{\frac{1}{2}}\big(x^{-1},y\big), \qquad
(f_1, f_2, f_3)\to h_1^{-1}(e_7, e_8, e_9), \qquad
(g_1,g_2)\to\bigg(\frac{1}{e_{11}},\frac{e_{10}}{h_2}\bigg),
\\
(h_1, \dots, h_6)\to e_{11}(e_1, \dots, e_6).
\end{gather*}
The corresponding tropical curve is a pencil of elliptic curves given in Figure~\ref{Fig:rect}.

\begin{figure}[tbp]
\begin{center}
\includegraphics[width=80mm]{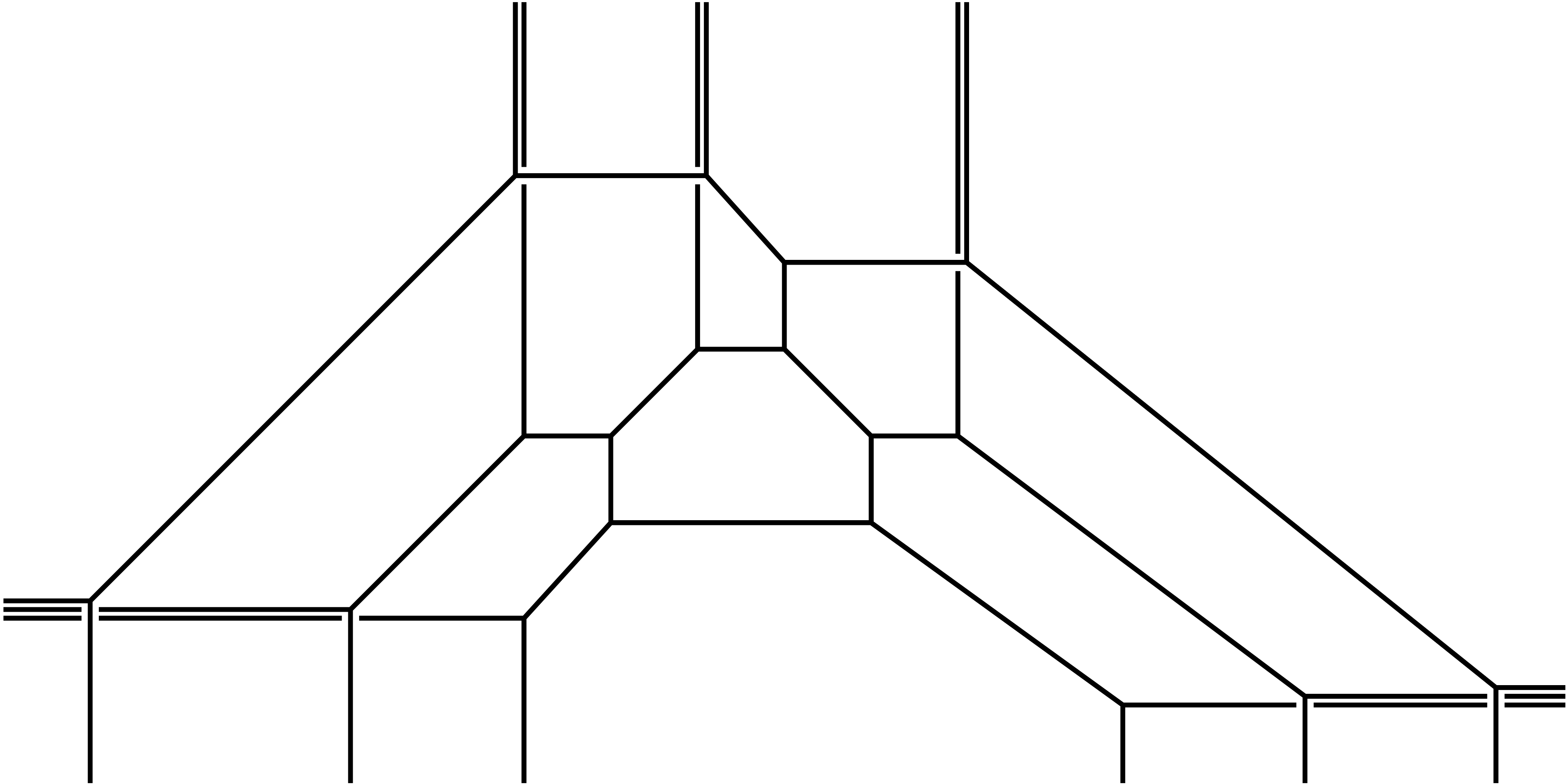}
\caption[Fig.2]{The web diagram corresponding to the curve $H(x,y)=E$ which has $6 \rm{(single)}+3 \rm{(double)}+2 \rm{(triple)}$ asymptotic lines.
It has one closed cycle whose size depends on the parameter $E$, hence \mbox{$\rm (genus)=1$}.
See, e.g.,~\cite{BBT,KTY,KY}.}
\label{Fig:rect}
\end{center}
\end{figure}

An explicit form of the polynomial $P_0(x,y)$ is given by
\begin{gather*}
P_0(x,y)=\sum_{i=0}^3 C_i(x)y^i,
\\
C_3(x)=q^3 e_{11}^3 \prod_{i=7}^9 \bigg(1+\frac{e_i}{h_1}x\bigg)\bigg(1+q\frac{e_i}{h_1}x\bigg),
\\
C_2(x)=q e_{11}^2\prod_{i=7}^9 \bigg(1+\frac{e_i}{h_1}x\bigg)
\big\{[3]_q+q x A_{-1}+q \kappa A_{1} x^2+[3]_q \kappa x^3\big\},
\\
C_1(x)=e_{11}\bigg\{[3]_q+[2]_q A_{-1} x+(\kappa A_1 +A_{-2})x^2
+\frac{\kappa}{q}(\kappa A_2+A_{-1})x^4
\\ \hphantom{C_1(x)=e_{11}\bigg\{}
{}+\frac{[2]_q \kappa^2 A_1}{q^2}x^5+\frac{[3]_q \kappa^2}{q^3} x^6\bigg\},
\\
 C_0(x)=\prod_{i=1}^6\bigg(1+\frac{1}{qe_i}x\bigg),
\end{gather*}
where
\begin{gather*}
A_{\pm 1}=\sum_{i=1}^9 a_i^{\pm 1}, \qquad
A_{\pm 2}=\sum_{1\leq i<j\leq 9} (a_i a_j)^{\pm 1},
\\
a_i=e_i\qquad (1\leq i\leq 6), \qquad
a_i=\frac{h_1}{e_i}\qquad (7\leq i\leq 9),
\\
[k]_q=\dfrac{1-q^k}{1-q},\qquad \kappa=\dfrac{e_7 e_8 e_9 e_{10} e_{11}}{h_1^2 h_2}.
\end{gather*}

\begin{Remark}
In the context of discrete Painlev\'e equations, the constraint (\ref{eq:ehcons}) gives the auto\-nomous case where the system admits a conserved curve $H(x,y)=E$.
\end{Remark}

\section{Bilinear equations}\label{section6}
The Weyl group representation in Theorem \ref{thm:quantum-rep} can be reformulated as follows.
\begin{Proposition}\label{prop:tau_Weyl}
Introduce variables $\tau_{1,10}$, $\tau_{1,11}$, $\tau_{2,1}$, $\tau_{2,7}$ instead of $x$, $y$, $\sigma_1$, $\sigma_2$ as
\begin{gather*}
\tau_{1,10}=\frac{\sigma_1}{\tau _{10}},\qquad
\tau_{1,11}= x\frac{\sigma_1}{\tau _{11}},\qquad
\tau_{2,1}= y\frac{\sigma_2}{\tau _1},\qquad
\tau_{2,7}= \frac{\sigma_2}{\tau _7}.
\end{gather*}
Then we have the representation of the Weyl group $W\big(E_8^{(1)}\big)$ given by
\begin{gather*}
s_0=\bigg\{e_{10}\to \dfrac{h_2}{e_{11}},\,
e_{11}\to \dfrac{h_2}{e_{10}},\,
h_1\to \dfrac{h_1h_2}{e_{10} e_{11}},
\\ \hphantom{s_0=\bigg\{}
\tau_{10}\to (\tau_{2,7}\tau_7+e_{11}\tau_{2,1}\tau_1)\dfrac{1}{\tau_{11}},\,
\tau_{11}\to \dfrac{1}{\tau_{10}}\bigg(\tau_{2,7}\tau_7+\dfrac{h_2}{e_{10}}\tau_{2,1}\tau_1\bigg)\bigg\},
\\
s_1=\{e_8\leftrightarrow e_9,\, \tau_8\leftrightarrow \tau_9\},\qquad
s_2=\bigg\{e_7\leftrightarrow e_8,\ \tau_7\leftrightarrow \tau_8,\, \tau_{2,7}\to \dfrac{\tau_7 \tau_{2,7}}{\tau_8}\bigg\},
\\
s_3=\bigg\{e_1\to \dfrac{h_1}{e_7},\,
e_7\to \dfrac{h_1}{e_1},\,
h_2\to \dfrac{h_1 h_2}{e_1 e_7},
\\ \hphantom{s_3=\bigg\{}
\tau_1\to \bigg(\tau_{1,10}\tau_{10}+\dfrac{e_7}{h_1}\tau_{1,11}\tau_{11}\bigg)\dfrac{1}{\tau_7},\
\tau_7\to \dfrac{1}{\tau_1}\bigg(\tau_{1,10}\tau_{10}+\dfrac{1}{e_1}\tau_{1,11}\tau_{11}\bigg)\bigg\},
\\
s_4=\bigg\{e_1\leftrightarrow e_2,\, \tau_1\leftrightarrow \tau_2,\, \tau_{2,1}\to \dfrac{\tau_1 \tau_{2,1}}{\tau_2}\bigg\},\quad
s_5=\{e_2\leftrightarrow e_3,\, \tau_2\leftrightarrow \tau_3\},
\\
s_6=\{e_3\leftrightarrow e_4,\, \tau_3\leftrightarrow \tau_4\},\qquad
s_7=\{e_4\leftrightarrow e_5,\, \tau_4\leftrightarrow \tau_5\},\qquad
s_8=\{e_5\leftrightarrow e_6,\, \tau_5\leftrightarrow \tau_6\}.
\end{gather*}
\end{Proposition}

\begin{proof} The actions written in the new variables are computed as follows.
\begin{gather*}
\tau_{10}\stackrel{s_0}{\to}(1+ye_{11})\frac{\sigma_2}{\tau_{11}}
=\bigg(\frac{\sigma_2}{\tau_7}\tau_7+e_{11}\frac{\sigma_2y}{\tau_1}\tau_1\bigg)\frac{1}{\tau_{11}}
=(\tau_{2,7}\tau_7+e_{11}\tau_{2,1}\tau_1)\frac{1}{\tau_{11}},
\\[1ex]
\tau_{11}\stackrel{s_0}{\to}\frac{\sigma_2}{\tau_{10}}\bigg(1+y\frac{h_2}{e_{10}}\bigg)
=\frac{1}{\tau_{10}}\bigg(\frac{\sigma_2}{\tau_7}\tau_7+\frac{h_2}{e_{10}}\frac{\sigma_2y}{\tau_1}\tau_1\bigg)
=\frac{1}{\tau_{10}}\bigg(\tau_{2,7}\tau_7+\frac{h_2}{e_{10}}\tau_{2,1}\tau_1\bigg),
\\[1ex]
\tau_1\stackrel{s_3}{\to}\bigg(1+x\frac{e_7}{h_1}\bigg)\frac{\sigma_1}{\tau_7}
=\bigg(\frac{\sigma_1}{\tau_{10}}\tau_{10}+\frac{e_7}{h_1}\frac{\sigma_1x}{\tau_{11}}\tau_{11}\bigg) \frac{1}{\tau_7}
=\bigg(\tau_{1,10}\tau_{10}+\frac{e_7}{h_1}\tau_{1,11}\tau_{11}\bigg)\frac{1}{\tau_7},
\\[1ex]
\tau_7\stackrel{s_3}{\to}\frac{\sigma_1}{\tau_1}\bigg(1+\frac{x}{e_1}\bigg)
=\frac{1}{\tau_1}\bigg(\frac{\sigma_1}{\tau_{10}}\tau_{10} +\frac{1}{e_1}\frac{\sigma_1x}{\tau_{11}}\tau_{11}\bigg)
=\frac{1}{\tau_1}\bigg(\tau_{1,10}\tau_{10}+\frac{1}{e_1}\tau_{1,11}\tau_{11}\bigg).
\end{gather*}
Other actions are obvious.
\end{proof}

In order to describe the bilinear equations in the Weyl-group covariant way, we define the tau functions $\tau(\lambda)$ on a certain lattice $L$ as follows.
\begin{itemize}\itemsep=0pt
\item[$(i)$] For $\lambda \in L_0=\big\{e_1, \dots, e_{11}, \frac{h_2}{e_1},\frac{h_2}{e_7},\frac{h_1}{e_{10}},\frac{h_1}{e_{11}}\big\}$,
we put
$\tau(e_i)=\tau_i$ ($1\leq i\leq 11$),
$\tau\big(\frac{h_2}{e_i}\big)=\tau_{2,i}$ ($i=1,7$) and
$\tau\big(\frac{h_1}{e_j}\big)=\tau_{1,j}$ ($j=10,11$).

\item[$(ii)$] For Weyl-group elements $w \in W\big(E_8^{(1)}\big)$, we put $\tau(w(\lambda))=w(\tau(\lambda))$.
\end{itemize}
From $(i)$ and $(ii)$, one can uniquely determine the functions $\tau(\lambda)$ for any $\lambda=\frac{h_1^{d_1}h_2^{d_2}}{e_1^{m_1}\cdots e_{11}^{m_{11}}} \in L$, where $L$ is the Weyl-group orbit of $L_0$.
For $\lambda\in\Lambda$, this fact is a consequence of Theorem~\ref{thm:sF}, and it can be extended for $\lambda\in L$ similarly by using the normalization condition
\begin{gather*}
\lim_{x\to 0}x^{-1}F_\lambda(x,y=0)=1,\qquad
{\rm or}\qquad
\lim_{y\to 0}y^{-1}F_\lambda(x=0,y)=1,
\end{gather*}
for $\lambda=w\big(\frac{h_1}{e_{11}}\big)$ or $\lambda=w\big(\frac{h_2}{e_1}\big)$, respectively.
(For the other cases we still have $F_{\lambda}(0,0)=1$.)

\begin{Corollary}\label{lemma:B-sol}
The functions $\tau(\lambda)$ satisfy the following relations:
\begin{gather}
\tau(e_{10})\tau\bigg(\dfrac{h_2}{e_{10}}\bigg)=\dfrac{h_2}{e_{10}} \tau\bigg(\dfrac{h_2}{e_i}\bigg) \tau(e_i)+ \tau\bigg(\dfrac{h_2}{e_j}\bigg)\tau(e_j),\nonumber
\\[1ex]
\tau\bigg(\dfrac{h_2}{e_{11}}\bigg)\tau(e_{11})=e_{11} \tau\bigg(\dfrac{h_2}{e_i}\bigg) \tau(e_i)+ \tau\bigg(\dfrac{h_2}{e_j}\bigg)\tau(e_j),\nonumber
\\[1ex]
\tau(e_i)\tau\bigg(\dfrac{h_1}{e_i}\bigg) =\dfrac{1}{e_i}\tau
\bigg(\dfrac{h_1}{e_{11}}\bigg)\tau(e_{11})+\tau\bigg(\dfrac{h_1}{e_{10}}\bigg)\tau(e_{10}),\nonumber
\\[1ex]
\tau\bigg(\dfrac{h_1}{e_j}\bigg)\tau(e_j) =\dfrac{e_j}{h_1}\tau\bigg(\dfrac{h_1}{e_{11}}\bigg)\tau(e_{11}) +\tau\bigg(\dfrac{h_1}{e_{10}}\bigg)\tau(e_{10}),\nonumber
\\[1ex]
\tau\bigg(\dfrac{h_2}{e_1}\bigg) \tau(e_1)=\dots= \tau\bigg(\dfrac{h_2}{e_6}\bigg) \tau(e_6),\nonumber
\qquad
\tau\bigg(\dfrac{h_2}{e_7}\bigg) \tau(e_7)=\dots= \tau\bigg(\dfrac{h_2}{e_9}\bigg) \tau(e_9),
\label{eq:seeds}
\end{gather}
where $1\leq i\leq 6$ and $7\leq j\leq 9$.
Furthermore, the infinitely many bilinear relations obtained from the above equations via the Weyl group actions also hold.
\end{Corollary}

\begin{proof} This is a simple reformulation of Proposition \ref{prop:tau_Weyl}. \end{proof}

\begin{Example}
The $s_0$ transform of the fourth equation in equation~(\ref{eq:seeds}) with $j=7$ is
\begin{gather}\label{eq:ex-bilinear}
\tau\bigg(\frac{h_1 h_2}{e_7 e_{10} e_{11}}\bigg)\tau(e_7)=
\frac{e_7e_{10}e_{11}}{h_1h_2}\tau\bigg(\frac{h_1}{e_{11}}\bigg)\tau\bigg(\frac{h_2}{e_{10}}\bigg)
+\tau\bigg(\frac{h_1}{e_{10}}\bigg)\tau\bigg(\frac{h_2}{e_{11}}\bigg).
\end{gather}
This can be confirmed by
\begin{gather*}
\tau\bigg(\dfrac{h_1 h_2}{e_7 e_{10} e_{11}}\bigg)
=\bigg\{1+e_{11}y+\dfrac{e_7 e_{10}e_{11}}{h_1 h_2}x\bigg(1+\dfrac{h_2}{e_{10}} y\bigg)\bigg\}
\dfrac{\sigma _1 \sigma _2 }{\tau _7 \tau _{10} \tau _{11}},\qquad
\tau(e_7)=\tau _7,
\\
\tau\bigg(\dfrac{h_1}{e_{11}}\bigg)=x\dfrac{\sigma_1}{\tau_{11}},\qquad
\tau\bigg(\dfrac{h_2}{e_{10}}\bigg)=\bigg(1+\dfrac{h_2}{qe_{10}}y\bigg)\dfrac{\sigma_2 }{\tau_{10}},
\\
\tau\bigg(\dfrac{h_1}{e_{10}}\bigg)=\dfrac{\sigma_1}{\tau_{10}},\qquad
\tau\bigg(\dfrac{h_2}{e_{11}}\bigg)=(1+e_{11}y)\dfrac{\sigma _2 }{\tau _{11}}.
\end{gather*}
Note that each term in equation~(\ref{eq:ex-bilinear}) is a member of two-parameter family $\tau\big(\frac{h_1 h_2}{e_{10} e_{11}}\big)$ (see equation~(\ref{eq:ex-two-para})) and should satisfy a relation among three of them.
\end{Example}

So far, we have derived the bilinear relations as the identities satisfied by the functions $\tau(\lambda)$ defined by the Weyl group actions.
Conversely, we can consider the relations as the infinite system of equations viewing $\tau(\lambda)$ $(\lambda \in L)$ as infinite unknown variables.
We call this overdetermined system of equations the {\it quantum bilinear equations} (denoted by $\mathcal B$).
Note that the bilinear system $\mathcal{B}$ is ``tropical" (or subtraction free)~\cite{Tsuda}.
\begin{Theorem}
For any initial data $\tau(\lambda)$ $(\neq 0)$ $(\lambda \in L_0)$, there exists a unique solution for the system $\mathcal{B}$.
And this solution gives the general solution.
\end{Theorem}

\begin{proof}
We already have a solution as given in Corollary \ref{lemma:B-sol}.
It has 15 free parameters $\tau_1, \dots, \tau_{11}$, $\sigma_1, \sigma_2$, $x, y$ which are enough to fit the 15 initial data $\tau(\lambda)$ ($\lambda \in L_0$).
\end{proof}

\begin{Remark}
In the commutative case ($q=1$), the space of the solutions of the system $\mathcal{B}$ is of dimension $2$ (modulo rescaling of variables $\tau_i$, $\sigma_i$) and can be identified with the Okamoto space with coordinates $x$, $y$.
Hence the system $\mathcal{B}$ can be considered as a quantum analog of the Pl\"ucker embedding of the Okamoto space~\cite{JM}.
It will be interesting if the system $\mathcal{B}$ can be obtained from some infinite-dimensional quantum integrable hierarchies.
In view of this, we~note that one can eliminate variables $\tau\big(\frac{h_1}{e_i}\big)$ ($i=10,11$) from the third and fourth equations in Corollary~\ref{lemma:B-sol} to derive the bilinear equations in the standard Hirota--Miwa form (see~\cite{ORG}) such as
\begin{gather*}
\bigg(\dfrac{1}{e_1}-\dfrac{1}{e_2}\bigg)\tau(e_3)\tau\bigg(\dfrac{h_1}{e_3}\bigg)+
\bigg(\dfrac{1}{e_2}-\dfrac{1}{e_3}\bigg)\tau(e_1)\tau\bigg(\dfrac{h_1}{e_1}\bigg)+
\bigg(\dfrac{1}{e_3}-\dfrac{1}{e_1}\bigg)\tau(e_2)\tau\bigg(\dfrac{h_1}{e_2}\bigg)=0,
\\[1ex]
\bigg(\dfrac{1}{e_1}-\dfrac{1}{e_2}\bigg)\tau\bigg(\dfrac{h_1}{e_7}\bigg)\tau(e_7)+
\bigg(\dfrac{1}{e_2}-\dfrac{e_7}{h_1}\bigg)\tau(e_1)\tau\bigg(\dfrac{h_1}{e_1}\bigg)+
\bigg(\dfrac{e_7}{h_1}-\dfrac{1}{e_1}\bigg)\tau(e_2)\tau\bigg(\dfrac{h_1}{e_2}\bigg)=0.
\end{gather*}
\end{Remark}

\section{Summary and discussions}\label{section7}

In this paper, we studied the quantization of the affine Weyl group of type $E_8^{(1)}$ and obtained the following several results:

\begin{itemize}\itemsep=0pt
\item
A quantum (non-commutative) version of the affine Weyl group representation is formulated (Theorem \ref{thm:quantum-rep}).

\item
Its realization as adjoint actions is obtained (Theorem \ref{thm:adjoint}).

\item Fundamental polynomials arising from the representation are studied and its characterization is given
(Theorem \ref{thm:sF}).

\item The quantum curve for $E_8^{(1)}$ in~\cite{Moriyama} is rederived by the Weyl group symmetry (Section~\ref{section5}).

\item The quantum bilinear equations are obtained (Section~\ref{section6}).

\end{itemize}

Many of the results can be formulated similarly for the cases $D_5^{(1)}$, $E_6^{(1)}$ and $E_7^{(1)}$ as well.
Such results are summarized in Appendix~\ref{sectionA}.

One of our motivations for studying the quantum curves is the correspondence between spectral theories and topological strings, as observed in~\cite{GHM,HMMO,HMO,HMO3}.
Namely, the determinant of the spectral operator obtained from the quantum curve is described by the free energy of topological strings on the same geometry, which is captured by the period integrals.
After providing the quantum curves and their origins in the affine Weyl groups, we believe that there are many directions to pursue to deepen the correspondence.
Here we list some of future problems.

\begin{itemize}\itemsep=0pt

\item
Given a quantum curve, the study of the {\it spectral problem} is important.
Since the expres\-sion is very huge in the exceptional cases $E_n^{(1)}$, the Weyl group symmetry will play a~fundamental role to control them as discussed in~\cite{Moriyama}.
It is interesting to start with the study of matrix elements of the spectral operators as in~\cite{KMZ2}.

\item
After fixing the spectral operators, besides the spectral determinant, we can study various invariant or covariant quantities including the $F$-polynomials defined above.
We believe that the correspondence is clarified from their relations.

\item
In relation to the spectral problem mentioned above, computation of the quantum period integrals is also an interesting problem~\cite{ACDKV,ADKMV,Huang,HKRS,HSW,MM}.
Even in genus one cases they are technical challenges in particular for the fully massive $E_6$, $E_7$, $E_8$ cases.
Again, we~expect that the Weyl groups serve an important role in studying the periods~\cite{FMS,MNY}.

\item
It is of course an interesting future direction to generalize our characterization to the cases of spectral operators of higher genus to study the correspondence in~\cite{CGM,CGuM}.

\item
The tau functions for some classical (commutative) Painlev\'e equations have the general solutions in terms of the Nekrasov functions (the Kiev formula) in both differential and difference cases (see~\cite{GIL,JNS} for example).
Their extension to quantum Painlev\'e equations is an important problem~\cite{BGM,BS}.

\item
In the context of Painlev\'e equations, the quantum curve appears in two ways:
(i) as the conserved quantity for quantum autonomous Painlev\'e equations discussed in this paper, and
(ii) as a ceratin specialization of the Lax linear equation for classical non-autonomous Painlev\'e equations~\cite{BGG,NRY,Take}.
In order to clarify the relation between the Kiev formula and the partition function of topological strings~\cite{GHM,HMMO}, it will be important to understand the relation between $(i)$ and $(ii)$.

\item
There is a Lens generalization of the discrete Painlev\'e equation~\cite{KeY2} whose identification in the Sakai's classification is not clear so far.
It may be related to a quantization where~$q$ is a root of unity.

\item
The Weyl group symmetry (the iWeyl group) for the various (quantum) Seiberg--Witten curves was obtained (see~\cite{KP,NP} for example).
The Weyl-group actions considered in this paper are expected to be a realization of the iWeyl group.

\end{itemize}

\appendix

\section[Cases D5(1), E6(1) and E7(1)]
{Cases $\boldsymbol{D_5^{(1)}}$, $\boldsymbol{E_6^{(1)}}$ and $\boldsymbol{E_7^{(1)}}$}\label{sectionA}

Here we will give the results for the cases $D_5^{(1)}$, $E_6^{(1)}$ and $E_7^{(1)}$ which are similar\footnote{For the cases $D_5^{(1)}$, $E_6^{(1)}$, $E_7^{(1)}$ one can extend the affine Weyl group by including the automorphisms of the Dynkin diagram.
However we will not consider such extensions here.} to the case~$E_8^{(1)}$.
First, we prepare some notations which are common for all cases.

To describe the Weyl group actions, we put
\begin{gather*}
s_{i,j}=\{e_i \leftrightarrow e_j,\, \tau_i\leftrightarrow \tau_j\},
\\
s^{x}_{i,j}=\bigg\{e_{i}\to \dfrac{h_2}{e_{j}},\, e_{j}\to \dfrac{h_2}{e_{i}},\, h_1\to \dfrac{h_1 h_2}{e_{i} e_{j}}, \, x\to x\dfrac{1+y\frac{h_2}{e_{i}}}{1+ye_{j}},
\\ \hphantom{s^{x}_{i,j}=\bigg\{}
\tau_{i}\to \big(1+ye_{j}\big)\dfrac{\sigma_2}{\tau_{j}},\,
\tau_{j}\to \dfrac{\sigma_2}{\tau_{i}}\bigg(1+y\dfrac{h_2}{e_{i}}\bigg),\,
\sigma_1\to \big(1+ye_{j}\big)\dfrac{\sigma_1 \sigma_2}{\tau_{i} \tau_{j}} \bigg\},
\\
s^{y}_{i,j}=\bigg\{e_i\to \dfrac{h_1}{e_j},\, e_j\to \dfrac{h_1}{e_i},\, h_2\to \dfrac{h_1 h_2}{e_i e_j},\, y\to \dfrac{1+x\frac{e_j}{h_1}}{1+\frac{x}{e_i}}y,
\\ \hphantom{s^{y}_{i,j}=\bigg\{}
\tau_i\to \bigg(1+x\dfrac{e_j}{h_1}\bigg)\dfrac{\sigma_1}{\tau_j},\,
\tau_j\to \dfrac{\sigma_1}{\tau_i}\bigg(1+\dfrac{x}{e_i}\bigg),\,
\sigma_2\to \dfrac{\sigma_1\sigma_2}{\tau_i \tau_j}\bigg(1+\dfrac{x}{e_i}\bigg)\bigg\}.
\end{gather*}
To specify the form of the $F$-polynomials for a given data $((d_1,d_2), (m_1, m_2,\dots))$, we put
\begin{gather*}
F^{x}_{I,J}=
\sum_{i=0}^{d_1} x^i
\prod_{k\in J}\prod_{t=i}^{m_{k}-1}\big(1+q^te_{k}y\big)
\prod_{k\in I}\prod_{t=d_1-m_{k}}^{i-1}\bigg(1+q^t\frac{h_2}{e_{k}}y\bigg)
U_i(y),
\end{gather*}
where ${\rm deg}\,U_i(y)=d_2-\sum_{k\in I}(i-d_1+m_{k})_{+}-\sum_{k \in J}(m_{k}-i)_{+}$, and
\begin{gather*}
F^y_{I,J}=
\sum_{i=0}^{d_2}\prod_{k \in I}\prod_{t=i-m_k}^{-1}\bigg(1+q^t\frac{1}{e_k}x\bigg)
\prod_{k \in J}\prod_{t=0}^{m_k-d_2+i-1}\bigg(1+q^t\frac{e_k}{h_1}x\bigg)V_i(x) \, y^i,
\end{gather*}
where ${\rm deg}\,V_i(x)=d_1-\sum_{k\in I}(m_{k}-i)_{+}-\sum_{k\in J}(i-d_2+m_{k})_{+}$.
With this notation, the previous result for the $E_8^{(1)}$ case is given by
\begin{gather*}
s_0=s^x_{10,11},\qquad
s_1=s_{8,9},\qquad
s_2=s_{7,8},\qquad
s_3=s^y_{1,7},\qquad
s_4=s_{1,2},
\\
s_5=s_{2,3},\qquad
s_6=s_{3,4},\qquad
s_7=s_{4,5},\qquad
s_8=s_{5,6}.
\end{gather*}
Besides, the notation is applicable to all the other lower-rank cases.
All these results are consistent with the quantum curves and the Weyl actions given in~\cite{Moriyama}.

\medskip\noindent
{\it Case $E_7^{(1)}$}:
\begin{itemize}\itemsep=0pt
\item
The Weyl group $W\big(E_7^{(1)}\big)$ corresponding to the Dynkin diagram
\begin{gather*}
\begin{array}{@{}c@{\ }c@{\ }c@{\ }c@{\ }c@{\ }c@{\ }c@{\ }c@{\ }c@{\ }c@{\ }c@{\ }c@{\ }c@{\ }c@{\ }c}
&&&&&&s_0\\
&&&&&&|\\
s_1&\text{---}&s_2&\text{---}&s_3&\text{---}&s_4&\text{---}&s_5&\text{---}&s_6&\text{---}&s_7,
\end{array}
\end{gather*}
can be realized as
\begin{gather*}
s_0=s^{x}_{9, 10}, \qquad
s_1=s_{7, 8}, \qquad
s_2=s_{6, 7}, \qquad
s_3=s_{5, 6},
\\
s_4=s^{y}_{1, 5}, \qquad
s_5=s_{1, 2}, \qquad
s_6=s_{2, 3}, \qquad
s_7=s_{3, 4}.
\end{gather*}
\item
Defining conditions for the $F$-polynomials are given as follows.
If we collect terms with the same power of $x$, the $F$-polynomials take the form $F^{x}_{\{9\}, \{10\}}$, while if we collect terms with the same power of $y$, they take the form $F^y_{\{1,2,3,4\}, \{5,6,7,8\}}$.

\item Under the condition $\frac{h_1^4 h_2^2}{e_1\cdots e_8e_9^2e_{10}^2}=1$, we have the quantum curve
\begin{align*}
P_{E_7^{(1)}}={}&\prod _{i=1}^{4} \bigg(1+\frac{x}{q e_i}\bigg)+\bigg\{e_{10} (1+q)
+e_{10} \bigg(\sum _{i=5}^{8} \frac{e_i}{h_1}+\sum _{i=1}^{4} \frac{1}{e_i}\bigg)x+c x^2
\\
&+\kappa x^3 \bigg(\sum _{i=5}^{8} \frac{h_1}{e_i}+\sum_{i=1}^{4} e_i\bigg)
+\frac{\kappa}{q} (1+q) x^4\bigg\}y+e_{10}^2 q \prod _{i=5}^{8} \bigg(1+\frac{e_ix}{h_1}\bigg) y^2,
\end{align*}
where $\kappa=\frac{h_2}{q e_1e_2e_3e_4e_9}$ and $c\in{\mathbb C}$.
\end{itemize}

\medskip\noindent
{\it Case $E_6^{(1)}$}:
\begin{itemize}\itemsep=0pt
\item
The Weyl group $W\big(E_6^{(1)}\big)$ corresponding to the Dynkin diagram
\begin{gather*}
\begin{array}{@{}c@{\ }c@{\ }c@{\ }c@{\ }c@{\ }c@{\ }c@{\ }c@{\ }c@{\ }c@{\ }c@{\ }c@{\ }c@{\ }c@{\ }}
&&&&s_0\\
&&&&|\\
&&&&s_6\\
&&&&|\\
s_1&\text{---}&s_2&\text{---}&s_3&\text{---}&s_4&\text{---}&s_5,
\end{array}
\end{gather*}
can be realized as
\begin{gather*}
s_0=s_{8, 9}, \qquad
s_1=s_{5, 6}, \qquad
s_2=s_{4, 5}, \qquad
s_3=s^{y}_{1, 4},
\\
s_4=s_{1, 2}, \qquad
s_5=s_{2, 3}, \qquad
s_6=s^{x}_{7, 8}.
\end{gather*}
\item
The $F$-polynomials take respectively the form of $F^{x}_{\{7\}, \{8,9\}}$ and $F^y_{\{1,2,3\}, \{4,5,6\}}$ if we collect the same power of $x$ and $y$.

\item Under the condition $\frac{h_1^3 h_2^2}{e_1\cdots e_6e_7^2e_8e_9}=1$, we have the quantum curve
\begin{gather*}
P_{E_6^{(1)}}=
\bigg\{e_8+e_9+cx+\frac{h_2}{q e_1 e_2 e_3 e_7}\bigg(\sum _{i=4}^{6} \frac{h_1}{e_i}+\sum_{i=1}^{3} e_i\bigg)x^2+\frac{h_2 (1+q)}{q^2 e_1 e_2 e_3 e_7}x^3\bigg\}y
\\ \hphantom{P_{E_6^{(1)}}=}
+e_8 e_9 \prod_{i=4}^{6} \bigg(1+\frac{e_i x}{h_1}\bigg) y^2
+\prod_{i=1}^{3} \bigg(1+\frac{x}{q e_i}\bigg).
\end{gather*}
\end{itemize}

\noindent
{\it Case~$D_5^{(1)}$}:
\begin{itemize}\itemsep=0pt
\item
The Weyl group $W\big(D_5^{(1)}\big)$ corresponding to the Dynkin diagram
\begin{gather*}
\begin{array}{@{}c@{\ }c@{\ }c@{\ }c@{\ }c@{\ }c@{\ }c@{\ }c@{\ }c@{\ }c@{\ }c@{\ }c@{\ }c@{\ }c@{}}
&&s_0&&s_4\\
&&|&&|\\
s_1&\text{---}&s_2&\text{---}&s_3&\text{---}&s_5,
\end{array}
\end{gather*}
can be realized as
\begin{gather*}
s_0=s_{7, 8}, \qquad
s_1=s_{3, 4}, \qquad
s_2=s^{y}_{3, 7}, \qquad
s_3=s^{x}_{1, 5}, \qquad
s_4=s_{1, 2}, \qquad
s_5=s_{5, 6}.
\end{gather*}
\item
The $F$-polynomials take respectively the form of $F^{x}_{\{1,2\}, \{5,6\}}$ and $F^y_{\{7,8\}, \{3,4\}}$ if we collect the same power of $x$ and $y$.

\item Under the condition $\frac{h_1^2 h_2^2}{e_1\cdots e_8}=1$, we have the quantum curve
\begin{gather*}
P_{D_5^{(1)}}=\prod_{i=7}^8\bigg(1+\frac{x}{q e_i}\bigg)
+\bigg(e_5+e_6+c x+\frac{(e_1+e_2) h_2}{q e_1 e_2 e_7 e_8}x^2\bigg)y
+e_5 e_6 \prod_{i=3}^4\bigg(1+\frac{e_i x}{h_1}\bigg) y^2.
\end{gather*}
\end{itemize}

\section{Standard realizations in commutative case}\label{sectionB}

In Sakai's theory~\cite{Sakai:2001}, the geometry relevant for the 2nd order discrete/continuous Painlev\'e equations are classified as in the following list:
\def\rr{\rightarrow}
\def\se{\searrow}
\def\ne{\nearrow}
\arraycolsep=3pt

\[{
\begin{array}{lcccccccccccccccccc}
{\rm elliptic}\quad &E_8\\
&&&&&&&&&&&&&&&&A_1\\[-3mm]
&&&&&&&&&&&&&&&\ne\\
{\rm multiplicative}\quad
&E_8&\rr&E_7&\rr&E_6&\rr&D_5&\rr&A_4
&&\rr&A_{2+1}&\rr&A_{1+1}&\rr&A_1&\rr&A_0\\[5mm]
{\rm additive}\quad &E_8&\rr&E_7&\rr&E_6&&\rr&&D_4
&&\rr&A_3&\rr&A_{1+1}&\rr&A_1&\rr&A_0\\
&&&&&&&&&&&&&\se&&\se\\
&&&&&&&&&&&&&&A_2&\rr&A_1&\rr&A_0\\
\end{array}
}\]

\vskip-10mm
\setlength{\unitlength}{0.8mm}
\begin{picture}(100,15)(-100,-3)
\put(0,0){\line(1,0){94}}
\put(0,25){\line(1,0){94}}
\put(0,0){\line(0,1){25}}
\put(94,0){\line(0,1){25}}
\end{picture}

This list is the same as the degeneration scheme of the $E$-string.
The classes of ``elliptic'', ``multiplicative'' and ``additive'' mean the types of the difference equation and correspond to the gauge theories in 6D/5D/4D (see, e.g.,~\cite{BLMST,MY}).
The cases in the box admit the continuous flows (of the original Painlev\'e equation), and the relation between their Hamiltonians and the $D=4$, ${\rm SU}(2)$ Seiberg--Witten curves was observed in~\cite{KMNOY:pencil}.
Symbols $A_n$, $D_n$, $E_n$ represent the \text{types} of the symmetry (affine in the Painlev\'e equations) and correspond to the (non-affine) flavor symmetry of the gauge theory.\footnote{Since the gauge theories are associated with the autonomous limit of the Painlev\'e equations, the affine Weyl groups are reduced to the finite Weyl groups.}

There are two standard ways to realize the above geometry, namely $(i)$ nine-point blow-up of $\P^2$ or $(ii)$ eight-point blow-up of $\P^1\times\P^1$.
In the most generic case, these points determine an elliptic curve and we have the elliptic Painlev\'e equation.
Here we will give the multiplicative case in the realizations $(i)$ and $(ii)$ together with their relations.

\medskip
\noindent
{\it $(i)$ $\P^2$-realization}. Consider a parametrization of a point $p_3(u)=(x(u):y(u):1) \in \P^2$
\begin{gather}\label{eq:para_p2}
x(u)=u, \qquad y(u)=\dfrac{\epsilon_0}{u}-u^2 \qquad \big(u \in \P^1\big).
\end{gather}
The equations parametrize a cubic curve $C_3$ (with a node) given by
\begin{gather*}
\varphi_3(x,y)=x^3+x y-\epsilon_0=0.
\end{gather*}
The group structure of the curve $C_3$ is multiplicative, i.e., $3n$ points $p_3(u_i)$ ($i=1,\dots,3n$) are intersections of $C_3$ and a curve of degree $n$ iff $u_1\cdots u_{3n}=\epsilon_0^n$.
Hence, the blow-up of $\P^2$ at the nine points $p_3(\epsilon_i)$ has the elliptic fibration iff $\epsilon_1\cdots \epsilon_{9}=\epsilon_0^3$.

\medskip
\noindent
{\it $(ii)$ $\P^1\times \P^1$-realization}. Consider a parametrization of a point $p_{2,2}(u)=(f(u), g(u)) \in \P^1\times \P^1$
\begin{gather}\label{eq:para_p1p1}
f(u)=u+\frac{h_1}{u}, \qquad g(u)=u+\frac{h_2}{u}.
\end{gather}
The equations parametrize a bidegree (2,2) curve $C_{2,2}$ (with a node) given by
\begin{gather*}
\varphi_{2,2}(f,g)= \frac{(f-g)(h_2f-h_1g)}{h_1-h_2}+(h_1-h_2)=0.
\end{gather*}
This curve is also multiplicative; $N=2(m+n)$ points $p_{2,2}(u_i)$ ($i=1,\dots, N$) are intersections of $C_{2,2}$ and a curve of bidegree $(m,n)$ iff $u_1\cdots u_{N}=h_1^m h_2^n$.
Hence, the blow-up of $\P^1\times \P^1$ at~the eight points $p_{2,2}(e_i)$ has the elliptic fibration iff $e_1\cdots e_{8}=h_1^2h_2^2$.

\begin{Proposition}
The realizations $(i)$ and $(ii)$ are equivalent through the following birational symplectic transformation of variables $(x,y)$ and $(f,g)$ with the identification of parameters $(\epsilon_0 \dots, \epsilon_9)$ and $(h_1,h_2, e_1, \dots, e_8)$ given by
\begin{gather}
f=\dfrac{\frac{\epsilon_0}{\epsilon_1}-\epsilon_1 x-y}{x-\epsilon_1}, \qquad
g=\dfrac{\frac{\epsilon_0}{\epsilon_2}-\epsilon_2 x-y}{x-\epsilon_2}, \nonumber
\\
h_1=\dfrac{\epsilon_0}{\epsilon_1}, \qquad
h_2=\dfrac{\epsilon_0}{\epsilon_2},\qquad
e_1=\dfrac{\epsilon_0}{\epsilon_1\epsilon_2}, \qquad
e_i=\epsilon_{i+1} \qquad (i>1).\label{eq:p2p1p1rel}
\end{gather}
\end{Proposition}

\begin{proof}
The relation of the parameters $(\epsilon_0 \dots, \epsilon_9)$ and $(h_1,h_2, e_1, \dots, e_8)$ is invertible (it is a~``li\-near'' isomorphism written in multiplicative coordinates).
Also, by a direct computation, we~see that the transformation between $(x,y)$ and $(f,g)$ is birational with the indeterminate points $p_3(\epsilon_1), p_3(\epsilon_2) \in \P^2$ and $p_{2,2}(e_1) \in \P^1\times \P^1$.
It is easy to check the parameterizations (\ref{eq:para_p2}), (\ref{eq:para_p1p1}) and the curves $C_3$, $C_{2,2}$ are mapped to each other by the transformation (\ref{eq:p2p1p1rel}).
Since
\begin{gather}\label{eq:symp}
\omega:=\frac{\text{d}x\wedge\text{d}y}{\varphi_3(x,y)}=\frac{\text{d}f\wedge\text{d}g}{\varphi_{2,2}(f,g)},
\end{gather}
equation~(\ref{eq:p2p1p1rel}) gives a symplectic transformation w.r.t.\ this symplectic form.
\end{proof}

Using the transformation (\ref{eq:p2p1p1rel}), we can derive the actions of affine Weyl group $W\big(E^{(1)}_8\big)$.
\begin{Proposition}
There exists a unique birational symplectic representation of affine Weyl group $W\big(E^{(1)}_8\big)$ with the following properties:
\begin{enumerate}\itemsep=0pt
\item[$(i)$] In variables $(x,y,\epsilon_0 \dots, \epsilon_9)$, the action is given by
\begin{gather*}
s_0=\bigg\{\epsilon_0 \to \dfrac{\epsilon_0^2}{\epsilon_1\epsilon_2\epsilon_3},\,
\epsilon_1 \to \dfrac{\epsilon_0}{\epsilon_2\epsilon_3},\,
\epsilon_2 \to \dfrac{\epsilon_0}{\epsilon_3\epsilon_1},\,
\epsilon_3 \to \dfrac{\epsilon_0}{\epsilon_1\epsilon_2},\,
x\to \tilde{x},\ y\to \tilde{y}\bigg\},
\\[1ex]
s_i=\{\epsilon_i \leftrightarrow \epsilon_{i+1}\} \qquad (i=1, \dots, 8),
\end{gather*}
where $\tilde{x}$ and $\tilde{y}$ are certain rational functions of $(x,y)$.

\item[$(ii)$] In variables $(f,g,h_1,h_2, v_1, \dots, v_8)$, the action is given by
\begin{gather*}
\begin{array}{l}
s_0=\{e_1 \leftrightarrow e_2\}, \qquad
s_i=\{e_{i-1} \leftrightarrow e_{i}\} \qquad (i=3, \dots, 8),
\\[1ex]
s_1=\{h_1 \leftrightarrow h_2,\, f \leftrightarrow g\},\qquad
s_2=\bigg\{h_2 \to \dfrac{h_1h_2}{e_1e_2},\, e_1 \to \dfrac{h_1}{e_2},\, e_2\to\dfrac{h_1}{e_1},\, g\to \tilde{g}\bigg\},
\end{array}
\end{gather*}
where $\tilde{g}$ is a certain rational function of $(f,g)$.
\end{enumerate}
\end{Proposition}

\begin{proof}
In the $\P^2$ realization, we have obvious symmetries $s_i=\{\epsilon_i \leftrightarrow \epsilon_{i+1}\}$ ($i=1, \dots, 8$) which generate $\mathfrak{S}_9$, and also in the $\P^1\times \P^1$ realization we have $\mathfrak{S}_2 \times \mathfrak{S}_8=\langle s_1=\{h_1 \leftrightarrow h_2, f \leftrightarrow g\}\rangle \times \langle s_0=\{e_1 \leftrightarrow e_2\}, s_i=\{e_{i-1} \leftrightarrow e_{i}\} (i=3, \dots.8) \rangle$ (see Figure~\ref{fig:two-E8}).
\begin{figure}[h]
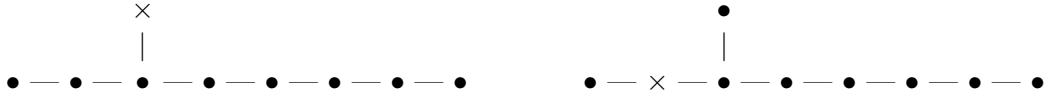

{\arraycolsep=2pt
$$
\begin{array}{@{}c@{\ }c@{\ }c@{\ }c@{\ }c@{\ }c@{\ }c@{\ }c@{\ }c@{\ }c@{\ }c@{\ }c@{\ }c@{\ }c@{\ }c@{}}
&&&&\times\\
&&&&|\\
\bullet&\text{---}&\bullet&\text{---}&\bullet&\text{---}&\bullet&\text{---}&\bullet&\text{---}&\bullet&\text{---}&\bullet&\text{---}&\bullet
\end{array}
\qquad\qquad
\begin{array}{@{}c@{\ }c@{\ }c@{\ }c@{\ }c@{\ }c@{\ }c@{\ }c@{\ }c@{\ }c@{\ }c@{\ }c@{\ }c@{\ }c@{\ }c@{}}
&&&&\bullet\\
&&&&|\\
\bullet&\text{---}&\times&\text{---}&\bullet&\text{---}&\bullet&\text{---}&\bullet&\text{---}&\bullet&\text{---}&\bullet&\text{---}&\bullet
\end{array}
$$
\caption{$\mathfrak{S}_9$ (left) and $\mathfrak{S}_2 \times \mathfrak{S}_8$ (right) subgroups in $W\big(E_8^{(1)}\big)$.}
\label{fig:two-E8}
}
\end{figure}
By mixing up the actions $\mathfrak{S}_9$ and $\mathfrak{S}_2\times\mathfrak{S}_8$, one can obtain the full generators for $W\big(E^{(1)}_8\big)$.
The non-trivial actions $s_0$ in~$(i)$ and $s_2$ in $(ii)$ can be obtained from the obvious actions in opposite realization through the relation (\ref{eq:p2p1p1rel}).
The explicit forms of $\tilde{x}$, $\tilde{y}$, $\tilde{g}$ can be determined by
\begin{gather*}
s_0(x)=\dfrac{x \epsilon_0(\epsilon_0-w)}{\epsilon_0^2-\epsilon_1\epsilon_2\epsilon_3w},\qquad
s_0(w)=w,\qquad
w=\dfrac{(x-\epsilon_1)(x-\epsilon_2)(x-\epsilon_3)}{\epsilon_0(x^3+x y-\epsilon_0)},
\end{gather*}
and
\begin{gather*}
s_2\Bigg(\frac{g-\big(v_1+\frac{h_2}{v_1}\big)}{g-\big(v_2+\frac{h_2}{v_2}\big)}\Bigg)
=\dfrac{f-\big(v_2+\frac{h_1}{v_2}\big)}{f-\big(v_1+\frac{h_1}{v_1}\big)}
\dfrac{g-\big(v_1+\frac{h_2}{v_1}\big)}{g-\big(v_2+\frac{h_2}{v_2}\big)}.
\end{gather*}
Thus, we obtain the desired results.
\end{proof}

\begin{Remark}
Written in the coordinates $(x,w)$, the Weyl group representation of $W\big(E_8^{(1)}\big)$ is the same as that in Section~\ref{section1} up to a change of the parameters (note that $w$ here corresponds to $y$ in Section~\ref{section1}).
In this coordinate, the symplectic form (\ref{eq:symp}) takes a simple form $\omega=\frac{\text{d}x\wedge\text{d}w}{x w}$.
This explains the reason why the realization in Section~\ref{section1} is suitable for quantization.
\end{Remark}

In closing this appendix, we will give explicit forms of the pencil of the conserved elliptic curves.

In the realization $(i)$ the conserved curve is given by
\begin{gather*}
0=\lambda \varphi_3+m_1 \big({-}x^2 \epsilon_0-x y^2+y \epsilon_0\big)
-m_2 x (x y-\epsilon_0)-m_3x^3+m_4 x^2-m_5 x
\\ \hphantom{0=}
{}+m_6
-m_7 \big(x^2+y\big){\epsilon_0}^{-1}
+m_8 \big(x^2 y+x \epsilon_0+y^2\big){\epsilon_0}^{-2}
-m_9 \big({-}3 x^3 \epsilon_0+y^3+3 \epsilon_0^2\big){\epsilon_0}^{-3},
\end{gather*}
where
$\sum_{i=0}^9 m_iz^i=\prod_{j=1}^9(1+\epsilon_i z)$
under the constraint $\epsilon_0^3=\epsilon_1\cdots \epsilon_9$.

In the realization $(ii)$ the conserved curve is given by
\begin{gather*}
0=\lambda \varphi_{2,2}
-m_1 \big[f g (g h_1-f h_2)+f h_2^2-gh_1^2\big]
+m_2\frac{(g h_1-f h_2){}^2}{h_1-h_2}-m_3 (g h_1-f h_2)
\\ \hphantom{0=}
{}+m_4(h_1-h_2)-m_5 (f-g)+m_6\frac{(f-g)^2}{h_1-h_2}
-m_7\frac{f g (f-g)-f h_1+gh_2}{h_1 h_2}
\\ \hphantom{0=}
{}+(h_1-h_2) \big(\big(f^2-2 h_1\big) \big(g^2-2 h_2\big)-h_1^2-h_2^2\big),
\end{gather*}
where
$\sum_{i=0}^8 m_iz^i=\prod_{j=1}^8(1+v_i z)$
under the constraint $h_1^2h_2^2=v_1\cdots v_8$.

Written in the Weierstrass form these curves coincide with the Seiberg--Witten curve for 5D $E$-string~\cite{ES03}.
In the quantum case, we do not know whether such cubic or bi-quadratic form is available or not so far.

\subsection*{Acknowledgements}
We would like thank to our colleagues for valuable discussions.
The work of S.M.\ is supported by~Grant-in-Aid for Scientific Research~(C) No.~19K03829.
The work of Y.Y.\ is supported by~Grant-in-Aid for Scientific Research~(S) No.~17H06127.

\pdfbookmark[1]{References}{ref}
\LastPageEnding


\begin{thebibliography}{99}
\footnotesize\itemsep=0pt

\bibitem{ACDKV}
Aganagic M., Cheng M.C.N., Dijkgraaf R., Krefl D., Vafa C., Quantum geometry of
 refined topological strings, \href{https://doi.org/10.1007/JHEP11(2012)019}{\textit{J.~High Energy Phys.}} \textbf{2012} (2012), no.~11,
 019, 53~pages, \href{https://arxiv.org/abs/1105.0630}{arXiv:1105.0630}.

\bibitem{ADKMV}
Aganagic M., Dijkgraaf R., Klemm A., Mari\~no M., Vafa C., Topological strings
 and integrable hierarchies, \href{https://doi.org/10.1007/s00220-005-1448-9}{\textit{Comm. Math. Phys.}} \textbf{261} (2006),
 451--516, \href{https://arxiv.org/abs/hep-th/0312085}{arXiv:hep-th/0312085}.

\bibitem{BBT}
Benini F., Benvenuti S., Tachikawa Y., Webs of five-branes and {${\mathcal
 N}=2$} superconformal field theories, \href{https://doi.org/10.1088/1126-6708/2009/09/052}{\textit{J.~High Energy Phys.}} \textbf{2009} (2009),
 no.~9, 052, 33~pages, \href{https://arxiv.org/abs/0906.0359}{arXiv:0906.0359}.

\bibitem{BGG}
Bershtein M., Gavrylenko P., Grassi A., Quantum spectral problems and
 isomonodromic deformations, \href{https://arxiv.org/abs/2105.00985}{arXiv:2105.00985}.

\bibitem{BGM}
Bershtein M., Gavrylenko P., Marshakov A., Cluster integrable systems,
 {$q$}-{P}ainlev\'e equations and their quantization, \href{https://doi.org/10.1007/jhep02(2018)077}{\textit{J.~High Energy
 Phys.}} \textbf{2018} (2018), no.~2, 077, 34~pages, \href{https://arxiv.org/abs/1711.02063}{arXiv:1711.02063}.

\bibitem{BS}
Bershtein M.A., Shchechkin A.I., {$q$}-deformed {P}ainlev\'e {$\tau$} function
 and {$q$}-deformed conformal blocks, \href{https://doi.org/10.1088/1751-8121/aa5572}{\textit{J.~Phys.~A: Math. Theor.}}
 \textbf{50} (2017), 085202, 22~pages, \href{https://arxiv.org/abs/1608.02566}{arXiv:1608.02566}.

\bibitem{BGT17}
Bonelli G., Grassi A., Tanzini A., Quantum curves and {$q$}-deformed
 {P}ainlev\'e equations, \href{https://doi.org/10.1007/s11005-019-01174-y}{\textit{Lett. Math. Phys.}} \textbf{109} (2019),
 1961--2001, \href{https://arxiv.org/abs/1710.11603}{arXiv:1710.11603}.

\bibitem{BLMST}
Bonelli G., Lisovyy O., Maruyoshi K., Sciarappa A., Tanzini A., On
 {P}ainlev\'e/gauge theory correspondence, \href{https://doi.org/10.1007/s11005-017-0983-6}{\textit{Lett. Math. Phys.}}
 \textbf{107} (2017), 2359--2413, \href{https://arxiv.org/abs/1612.06235}{arXiv:1612.06235}.

\bibitem{CMR}
Chekhov L., Mazzocco M., Rubtsov V., Quantised {P}ainlev\'e monodromy
 manifolds, {S}klyanin and {C}alabi--{Y}au algebras, \href{https://doi.org/10.1016/j.aim.2020.107442}{\textit{Adv. Math.}}
 \textbf{376} (2021), 107442, 52~pages, \href{https://arxiv.org/abs/2020.10744}{arXiv:2020.10744}.

\bibitem{CHKSW}
Chen J., Haghighat B., Kim H.-C., Sperling M., Wang X., $E$-string quantum
 curve, \href{https://arxiv.org/abs/2103.16996}{arXiv:2103.16996}.

\bibitem{CGM}
Codesido S., Grassi A., Mari\~no M., Spectral theory and mirror curves of
 higher genus, \href{https://doi.org/10.1007/s00023-016-0525-2}{\textit{Ann. Henri Poincar\'e}} \textbf{18} (2017), 559--622,
 \href{https://arxiv.org/abs/1507.02096}{arXiv:1507.02096}.

\bibitem{CGuM}
Codesido S., Gu J., Mari\~no M., Operators and higher genus mirror curves,
 \href{https://doi.org/10.1007/JHEP02(2017)092}{\textit{J.~High Energy Phys.}} \textbf{2017} (2017), no.~2, 092, 53~pages,
 \href{https://arxiv.org/abs/1609.00708}{arXiv:1609.00708}.

\bibitem{ES03}
Eguchi T., Sakai K., Seiberg--{W}itten curve for {$E$}-string theory revisited,
 \href{https://doi.org/10.4310/ATMP.2003.v7.n3.a3}{\textit{Adv. Theor. Math. Phys.}} \textbf{7} (2003), 419--455,
 \href{https://arxiv.org/abs/hep-th/0211213}{arXiv:hep-th/0211213}.

\bibitem{FZ}
Fomin S., Zelevinsky A., The {L}aurent phenomenon, \href{https://doi.org/10.1006/aama.2001.0770}{\textit{Adv. in Appl. Math.}}
 \textbf{28} (2002), 119--144, \href{https://arxiv.org/abs/math.CO/0104241}{arXiv:math.CO/0104241}.

\bibitem{FMS}
Furukawa T., Moriyama S., Sugimoto Y., Quantum mirror map for del {P}ezzo
 geometries, \href{https://doi.org/10.1088/1751-8121/ab93fe}{\textit{J.~Phys.~A: Math. Theor.}} \textbf{53} (2020), 385401,
 18~pages, \href{https://arxiv.org/abs/1908.11396}{arXiv:1908.11396}.

\bibitem{GIL}
Gamayun O., Iorgov N., Lisovyy O., Conformal field theory of {P}ainlev\'e~{VI},
 \href{https://doi.org/10.1007/JHEP10(2012)038}{\textit{J.~High Energy Phys.}} \textbf{2012} (2012), no.~10, 038, 24~pages,
 {E}rratum, \href{https://doi.org/10.1007/JHEP10(2012)183}{\textit{J.~High
 Energy Phys.}} \textbf{2012} (2012), no.~10, 183, 1~page, \href{https://arxiv.org/abs/1207.0787}{arXiv:1207.0787}.

\bibitem{GNR}
Grammaticos B., Nijhoff F.W., Ramani A., Discrete {P}ainlev\'e equations, in
 The {P}ainlev\'e Property, \textit{CRM Ser. Math. Phys.}, \href{https://doi.org/10.1007/978-1-4612-1532-5_7}{Springer}, New York, 1999,
 413--516.

\bibitem{GHM}
Grassi A., Hatsuda Y., Mari\~no M., Topological strings from quantum mechanics,
 \href{https://doi.org/10.1007/s00023-016-0479-4}{\textit{Ann. Henri Poincar\'e}} \textbf{17} (2016), 3177--3235,
 \href{https://arxiv.org/abs/1410.3382}{arXiv:1410.3382}.

\bibitem{Hasegawa}
Hasegawa K., Quantizing the {B}\"{a}cklund transformations of {P}ainlev\'e
 equations and the quantum discrete {P}ainlev\'e {VI} equation, in Exploring
 New Structures and Natural Constructions in Mathematical Physics,
 \textit{Adv. Stud. Pure Math.}, Vol.~61, \href{https://doi.org/10.2969/aspm/06110275}{Math. Soc. Japan}, Tokyo, 2011,
 275--288, \href{https://arxiv.org/abs/math.QA/0703036}{arXiv:math.QA/0703036}.

\bibitem{HMMO}
Hatsuda Y., Mari\~no M., Moriyama S., Okuyama K., Non-perturbative effects and
 the refined topological string, \href{https://doi.org/10.1007/JHEP09(2014)168}{\textit{J.~High Energy Phys.}} \textbf{2014}
 (2014), no.~9, 168, 42~pages, \href{https://arxiv.org/abs/1306.1734}{arXiv:1306.1734}.

\bibitem{HMO}
Hatsuda Y., Moriyama S., Okuyama K., Instanton effects in {ABJM} theory from
 {F}ermi gas approach, \href{https://doi.org/10.1007/JHEP01(2013)158}{\textit{J.~High Energy Phys.}} \textbf{2013} (2013),
 no.~1, 158, 40~pages, \href{https://arxiv.org/abs/1211.1251}{arXiv:1211.1251}.

\bibitem{HMO3}
Hatsuda Y., Moriyama S., Okuyama K., Instanton bound states in {ABJM} theory,
 \href{https://doi.org/10.1007/JHEP05(2013)054}{\textit{J.~High Energy Phys.}} \textbf{2013} (2013), no.~5, 054, 22~pages,
 \href{https://arxiv.org/abs/1301.5184}{arXiv:1301.5184}.

\bibitem{Huang}
Huang M.-X., On gauge theory and topological string in {N}ekrasov--{S}hatashvili
 limit, \href{https://doi.org/10.1007/JHEP06(2012)152}{\textit{J.~High Energy Phys.}} \textbf{2012} (2012), no.~6, 152, 40~pages,
 \href{https://arxiv.org/abs/1205.3652}{arXiv:1205.3652}.

\bibitem{HKRS}
Huang M.-X., Klemm A., Reuter J., Schiereck M., Quantum geometry of del {P}ezzo
 surfaces in the {N}ekrasov--{S}hatashvili limit, \href{https://doi.org/10.1007/JHEP02(2015)031}{\textit{J.~High Energy
 Phys.}} \textbf{2015} (2015), no.~2, 031, 52~pages, \href{https://arxiv.org/abs/1401.4723}{arXiv:1401.4723}.

\bibitem{HSW}
Huang M.-X., Sugimoto Y., Wang X., Quantum periods and spectra in dimer models
 and {C}alabi--{Y}au geometries, \href{https://doi.org/10.1007/jhep09(2020)168}{\textit{J.~High Energy Phys.}} \textbf{2020}
 (2020), no.~9, 168, 36~pages, \href{https://arxiv.org/abs/2006.13482}{arXiv:2006.13482}.

\bibitem{JM}
Jimbo M., Miwa T., Monodromy preserving deformation of linear ordinary
 differential equations with rational coefficients.~{II}, \href{https://doi.org/10.1016/0167-2789(81)90021-X}{\textit{Phys.~D}}
 \textbf{2} (1981), 407--448.

\bibitem{JNS}
Jimbo M., Nagoya H., Sakai H., C{FT} approach to the {$q$}-{P}ainlev\'e {VI}
 equation, \href{https://doi.org/10.1093/integr/xyx009}{\textit{J.~Integrable Syst.}} \textbf{2} (2017), xyx009, 27~pages,
 \href{https://arxiv.org/abs/1706.01940}{arXiv:1706.01940}.

\bibitem{KMNOY:pencil}
Kajiwara K., Masuda T., Noumi M., Ohta Y., Yamada Y., Cubic pencils and
 {P}ainlev\'e {H}amiltonians, \href{https://doi.org/10.1619/fesi.48.147}{\textit{Funkcial. Ekvac.}} \textbf{48} (2005),
 147--160, \href{https://arxiv.org/abs/nlin.SI/0403009}{arXiv:nlin.SI/0403009}.

\bibitem{KNY}
Kajiwara K., Noumi M., Yamada Y., Geometric aspects of {P}ainlev\'e equations,
 \href{https://doi.org/10.1088/1751-8121/50/7/073001}{\textit{J.~Phys.~A: Math. Theor.}} \textbf{50} (2017), 073001, 164~pages,
 \href{https://arxiv.org/abs/1509.08186}{arXiv:1509.08186}.

\bibitem{KMZ2}
Kashaev R., Mari\~no M., Zakany S., Matrix models from operators and
 topological strings,~2, \href{https://doi.org/10.1007/s00023-016-0471-z}{\textit{Ann. Henri Poincar\'e}} \textbf{17} (2016),
 2741--2781, \href{https://arxiv.org/abs/1505.02243}{arXiv:1505.02243}.

\bibitem{KeY2}
Kels A.P., Yamazaki M., Lens generalisation of {$\tau$}-functions for the
 elliptic discrete {P}ainlev\'e equation, \href{https://doi.org/10.1093/imrn/rnz063}{\textit{Int. Math. Res. Not.}}
 \textbf{2021} (2021), 110--151, \href{https://arxiv.org/abs/1810.12103}{arXiv:1810.12103}.

\bibitem{KTY}
Kim S.-S., Taki M., Yagi F., Tao probing the end of the world, \href{https://doi.org/10.1093/ptep/ptv108}{\textit{Prog.
 Theor. Exp. Phys.}} \textbf{2015} (2015), 083B02, 29~pages,
 \href{https://arxiv.org/abs/1504.03672}{arXiv:1504.03672}.

\bibitem{KY}
Kim S.-S., Yagi F., 5d {$E_n$} {S}eiberg--{W}itten curve via toric-like diagram,
 \href{https://doi.org/10.1007/JHEP06(2015)082}{\textit{J.~High Energy Phys.}} \textbf{2015} (2015), no.~6, 082, 62~pages,
 \href{https://arxiv.org/abs/1411.7903}{arXiv:1411.7903}.

\bibitem{KP}
Kimura T., Pestun V., Twisted reduction of quiver W-algebrass,
 \href{https://arxiv.org/abs/1905.03865}{arXiv:1905.03865}.

\bibitem{Kirillov}
Kirillov A.N., Dilogarithm identities, \href{https://doi.org/10.1143/PTPS.118.61}{\textit{Progr. Theoret. Phys. Suppl.}} \textbf{118} (1995), 61--142, \href{https://arxiv.org/abs/hep-th/9408113}{arXiv:hep-th/9408113}.

\bibitem{KMN}
Kubo N., Moriyama S., Nosaka T., Symmetry breaking in quantum curves and super
 {C}hern--{S}imons matrix models, \href{https://doi.org/10.1007/jhep01(2019)210}{\textit{J.~High Energy Phys.}} \textbf{2019}
 (2019), no.~1, 210, 29~pages, \href{https://arxiv.org/abs/1811.06048}{arXiv:1811.06048}.

\bibitem{Kuroki}
Kuroki G., Quantum groups and quantization of {W}eyl group symmetries of
 {P}ainlev\'e systems, in Exploring New Structures and Natural Constructions
 in Mathematical Physics, \textit{Adv. Stud. Pure Math.}, Vol.~61, \href{https://doi.org/10.2969/aspm/06110289}{Math. Soc.
 Japan}, Tokyo, 2011, 289--325, \href{https://arxiv.org/abs/0808.2604}{arXiv:0808.2604}.

\bibitem{Kuroki2}
Kuroki G., Regularity of quantum tau-functions generated by quantum birational
 {W}eyl group actions, \href{https://arxiv.org/abs/1206.3419}{arXiv:1206.3419}.

\bibitem{MM}
Mironov A., Morozov A., Nekrasov functions from exact {B}ohr--{S}ommerfeld
 periods: the case of {${\rm SU}(N)$}, \href{https://doi.org/10.1088/1751-8113/43/19/195401}{\textit{J.~Phys.~A: Math. Theor.}}
 \textbf{43} (2010), 195401, 11~pages, \href{https://arxiv.org/abs/0911.2396}{arXiv:0911.2396}.

\bibitem{MY}
Mizoguchi S., Yamada Y., {$W(E_{10})$} symmetry, {M}-theory and {P}ainlev\'e
 equations, \href{https://doi.org/10.1016/S0370-2693(02)01870-1}{\textit{Phys. Lett.~B}} \textbf{537} (2002), 130--140,
 \href{https://arxiv.org/abs/hep-th/0202152}{arXiv:hep-th/0202152}.

\bibitem{Moriyama}
Moriyama S., Spectral theories and topological strings on del {P}ezzo
 geometries, \href{https://doi.org/10.1007/jhep10(2020)154}{\textit{J.~High Energy Phys.}} \textbf{2020} (2020), no.~10, 154,
 53~pages, \href{https://arxiv.org/abs/2007.05148}{arXiv:2007.05148}.

\bibitem{MNY}
Moriyama S., Nosaka T., Yano K., Superconformal {C}hern--{S}imons theories from
 del {P}ezzo geometries, \href{https://doi.org/10.1007/jhep11(2017)089}{\textit{J.~High Energy Phys.}} \textbf{2017} (2017),
 no.~11, 089, 40~pages, \href{https://arxiv.org/abs/1707.02420}{arXiv:1707.02420}.

\bibitem{NagY}
Nagao H., Yamada Y., Study of {$q$}-{G}arnier system by {P}ad\'e method,
 \href{https://doi.org/10.1619/fesi.61.109}{\textit{Funkcial. Ekvac.}} \textbf{61} (2018), 109--133, \href{https://arxiv.org/abs/1601.01099}{arXiv:1601.01099}.

\bibitem{NP}
Nekrasov N., Pestun V., Seiberg--{W}itten geometry of four dimensional
 {${\mathcal N}=2$} quiver gauge theories, \href{https://arxiv.org/abs/1211.2240}{arXiv:1211.2240}.

\bibitem{NRY}
Noumi M., Ruijsenaars S., Yamada Y., The elliptic {P}ainlev\'e {L}ax equation
 vs. van {D}iejen's 8-coupling elliptic {H}amiltonian, \href{https://doi.org/10.3842/SIGMA.2020.063}{\textit{SIGMA}}
 \textbf{16} (2020), 063, 16~pages, \href{https://arxiv.org/abs/1903.09738}{arXiv:1903.09738}.

\bibitem{NY:cmp}
Noumi M., Yamada Y., Affine {W}eyl groups, discrete dynamical systems and
 {P}ainlev\'e equations, \href{https://doi.org/10.1007/s002200050502}{\textit{Comm. Math. Phys.}} \textbf{199} (1998),
 281--295, \href{https://arxiv.org/abs/math.QA/9804132}{arXiv:math.QA/9804132}.

\bibitem{NY:birat}
Noumi M., Yamada Y., Birational {W}eyl group action arising from a nilpotent
 {P}oisson algebra, in Physics and Combinatorics 1999 ({N}agoya), \href{https://doi.org/10.1142/9789812810199_0010}{World Sci.
 Publ.}, River Edge, NJ, 2001, 287--319, \href{https://arxiv.org/abs/math.QA/0012028}{arXiv:math.QA/0012028}.

\bibitem{ORG}
Ohta Y., Ramani A., Grammaticos B., An affine {W}eyl group approach to the
 eight-parameter discrete {P}ainlev\'e equation, \href{https://doi.org/10.1088/0305-4470/34/48/316}{\textit{J.~Phys.~A: Math.
 Gen.}} \textbf{34} (2001), 10523--10532.

\bibitem{Sakai:2001}
Sakai H., Rational surfaces associated with affine root systems and geometry of
 the {P}ainlev\'e equations, \href{https://doi.org/10.1007/s002200100446}{\textit{Comm. Math. Phys.}} \textbf{220} (2001),
 165--229.

\bibitem{Take}
Takemura K., On {$q$}-deformations of the {H}eun equation, \href{https://doi.org/10.3842/SIGMA.2018.061}{\textit{SIGMA}}
 \textbf{14} (2018), 061, 16~pages, \href{https://arxiv.org/abs/1712.09564}{arXiv:1712.09564}.

\bibitem{Tsuda}
Tsuda T., Tropical {W}eyl group action via point configurations and
 {$\tau$}-functions of the {$q$}-{P}ainlev\'e equations, \href{https://doi.org/10.1007/s11005-006-0052-z}{\textit{Lett. Math.
 Phys.}} \textbf{77} (2006), 21--30.

\bibitem{Yamada17}
Yamada Y., An elliptic {G}arnier system from interpolation, \href{https://doi.org/10.3842/SIGMA.2017.069}{\textit{SIGMA}}
 \textbf{13} (2017), 069, 8~pages, \href{https://arxiv.org/abs/1706.05155}{arXiv:1706.05155}.

\end{thebibliography}
\end{document}